\documentclass[preprint]{imsart}
\RequirePackage[OT1]{fontenc}
\RequirePackage{amsthm,amsmath, amssymb, latexsym}
\RequirePackage[numbers, sort&compress]{natbib}
\RequirePackage[colorlinks,citecolor=blue,urlcolor=blue]{hyperref}
\usepackage{comment}
\usepackage{graphicx}
\usepackage{bbm}
\usepackage{color}
\usepackage{relsize}
\usepackage{csquotes}

\arxiv{arXiv:0000.0000}

\usepackage[justification=centering]{caption}
\usepackage{geometry}
\newcommand{\bE}{\mathbb{E}}

\newcommand{\bN}{\mathbb{N}}
\newcommand{\bP}{\mathbb{P}}
\newcommand{\bR}{\mathbb{R}}
\newcommand{\bT}{\mathbb{T}}
\newcommand{\bU}{\mathbb{U}}

\newcommand{\cZ}{\mathcal{Z}}

\newcommand{\cB}{\mathcal{B}}

\newcommand{\cP}{\mathcal{P}}
\newcommand{\cT}{\mathcal{T}}

\startlocaldefs
\numberwithin{equation}{section}
\theoremstyle{plain}
\newtheorem{theorem}{Theorem}[section]
\theoremstyle{remark}

\theoremstyle{definition}
\newtheorem{definition}[theorem]{Definition}

\theoremstyle{theorem}
\newtheorem{corollary}[theorem]{Corollary}
\newtheorem{example}[theorem]{Example}
\theoremstyle{theorem}
\newtheorem{lemma}[theorem]{Lemma}
\theoremstyle{theorem}
\newtheorem{prop}[theorem]{Proposition}
\theoremstyle{theorem}

\endlocaldefs

\begin{document}

\begin{frontmatter}

\title{Recursive construction of continuum random trees}
%\atltitle{}

\runtitle{Recursive construction of CRTs}

\begin{aug}

\author{\fnms{Franz} \snm{Rembart}\thanksref{T1}\ead[label=e1]{franz.rembart@stats.ox.ac.uk}}
\and
\author{\fnms{Matthias} \snm{Winkel}\thanksref{T2}\ead[label=e2]{winkel@stats.ox.ac.uk}}
\address{Department of Statistics, University of Oxford, 24-29 St Giles, Oxford OX1 3LB, United Kingdom \\ \printead{e1,e2}}
\today
\runauthor{F. Rembart and M. Winkel}

\affiliation{University of Oxford}
\thankstext{T1}{Supported by EPSRC grant EP/P505666/1}
\thankstext{T2}{Supported by EPSRC grant EP/K029797/1}

\end{aug}

\begin{abstract}
We introduce a general recursive method to construct continuum random trees (CRTs) from independent copies of a random string of beads, that is, any random interval equipped 
with a random discrete probability measure, and from related structures. We prove the existence of these CRTs as a new application of the fixpoint method for recursive
distribution equations formalised in high generality by Aldous and Bandyopadhyay. 

We apply this recursive method to show the convergence to CRTs of various tree growth 
processes. We note alternative constructions of existing self-similar CRTs in the sense of Haas, Miermont and Stephenson, and we give for the first time 
constructions of random compact $\mathbb R$-trees that describe the genealogies of Bertoin's self-similar growth fragmentations. In forthcoming work, we develop 
further applications to embedding problems for CRTs, providing a binary embedding of the stable line-breaking construction that solves an open problem of Goldschmidt and 
Haas.\vspace{-0.2cm}

\end{abstract}

%\begin{abstract}[language=french]
%\end{abstract}

\begin{keyword}
\kwd{string of beads}
\kwd{$\mathbb{R}$-tree}
\kwd{continuum random tree}
\kwd{self-similar tree}
\kwd{stable tree}
\kwd{recursive distribution equation}
\kwd{tree growth process}
\kwd{growth fragmentation}
\kwd{Hausdorff dimension\vspace{-0.2cm}}

\end{keyword}

\end{frontmatter}

%\tableofcontents

%{\tt xxx Draft between authors of work in progress. Not for general distribution.}

\section{Introduction}
\label{Intro}

We introduce a new recursive method to construct continuum random trees (CRTs) from independent copies of a random string of beads, that is, any random interval equipped with a random discrete probability measure. Our construction is based on the concept of a recursive tree framework as formalised by Aldous and Bandyopadhyay \cite{14} to unify various constructions that relate to some \textit{discrete} branching structure of potentially infinite depth, but not a priori to construct CRTs. This construction method allows us to go beyond the class of self-similar trees introduced and constructed by Haas, Miermont and Stephenson \cite{33,34}, and in particular to generalise the bead-splitting processes of \cite{37}. 

Following \cite{20}, we call a separable metric space $(T,d)$ an \textit{$\mathbb R$-tree} if any $x,y \in T$ are connected by a unique injective path $[[x,y]] \subset T$ and if this path is isometric to the interval $[0, d(x,y)]$. The $\mathbb R$-trees in this paper are equipped with a distinguished \textit{root} vertex $\rho \in T$. We also consider \textit{weighted $\mathbb R$-trees} $(T, d, \rho, \mu)$ further equipped with a probability measure $\mu$ on the Borel sets $\mathcal B(T)$ of $(T,d)$. \textit{Continuum random trees} (CRTs) are random variables with values in a space of continuum trees, where a \textit{continuum tree} is a weighted $\mathbb R$-tree $(T,d,\rho,\mu)$ whose probability measure $\mu$ is supported by the set of leaves of $T$, has no atoms, and assigns positive mass to all subtrees above $x$ for each non-leaf $x \in T$. For $\beta,c\in(0,\infty)$, we consider the tree $(T,c^\beta d,\rho,c\mu)$ with all distances scaled by $c^\beta$ and masses scaled by $c$. 

We build a random $\mathbb R$-tree $\mathcal T$ from a random string of beads $\xi$ and rescaled i.i.d. random $\mathbb R$-trees $(\mathcal T_i,i\geq 1)$, as follows. The tree $\mathcal T$ is the output of a map $\phi_\beta$, taking $(\mathcal T_i, i \geq 1)$ and attaching these trees to the locations of the atoms of $\xi$, with lengths rescaled by the $\beta$-power of the respective atom mass for some $\beta \in (0, \infty)$,
\begin{equation} \mathcal T:=\phi_\beta \left(\xi, \mathcal T_i, i \geq 1\right). \label{AB1} \end{equation}
See Figure \ref{fig1}.
\begin{center}
\begin{figure}[h]\vspace{-0.9cm}

\includegraphics[scale=0.35]{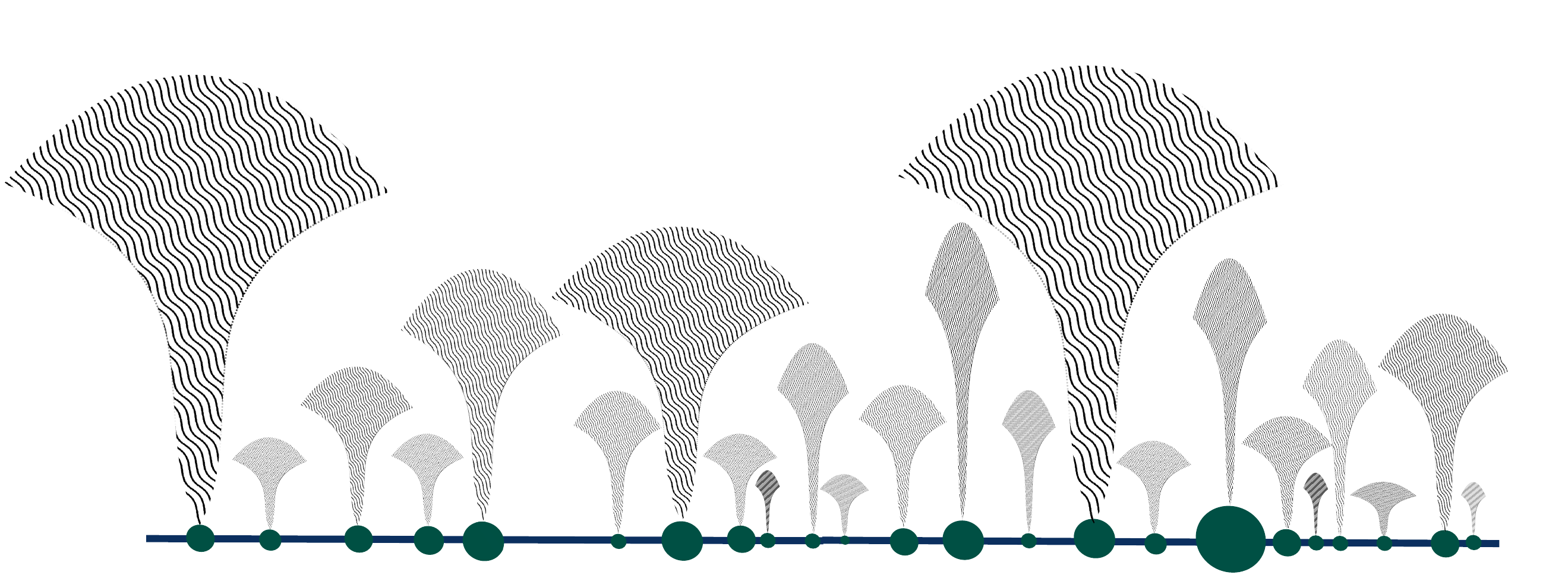}
\put(0,0){\normalsize $\xi$}
\put(0,30){\normalsize rescaled $\mathcal{T}_i$}
\caption{Rescaled trees $\mathcal{T}_i$, $i\ge 1$, attached to locations of the atoms of a string of beads $\xi$ to form $\phi_\beta(\xi,\mathcal{T}_i,i\ge 1)$.}\label{fig1}\vspace{-0.2cm}
\end{figure}
\end{center}
%$\;$%\pagebreak

\vspace{-0.65cm}

\noindent Denote by $\mathcal P(\mathbb T)$ the space of probability measures on the space $\mathbb T$ of (isometry classes of) compact $\mathbb R$-trees. Given (the distribution of) a random string of beads $\xi$ on a space of strings of beads, \eqref{AB1} yields a map $\Phi_\beta$ from $\mathcal P(\mathbb T)$ to $\mathcal P(\mathbb T)$, which associates with the 
common distribution of $\mathcal T_i$, $i\geq 1$, the distribution of $\phi_\beta(\xi,\mathcal T_i,i\geq 1)$. Similarly, we can also interpret \eqref{AB1} as a map from the space $\mathbb T_{\rm w}$ of (equivalence classes of) compact weighted $\mathbb R$-trees to $\mathbb T_{\rm w}$. See Section \ref{Sec31} for a more detailed introduction to $\mathbb R$-trees and CRTs, and Section \ref{rtp} for a formal definition of $\phi_\beta$ as a measurable function.   

Following \cite{14}, we show the existence of a fixpoint distribution in $\mathcal P(\mathbb T)$ for any random string of beads $\xi$. 
%Those are mainly based on a contraction property of the Wasserstein metric on a suitable subspace $\mathcal P \subset \mathcal P(\mathbb T)$. 
This establishes the existence of a large family of new CRTs, also including all so-called self-similar CRTs (and weighted $\mathbb R$-trees), which have been constructed differently by Haas, Miermont and Stephenson \cite{33,34}. This allows us to prove the convergence to such a CRT of various tree growth procedures based on an i.i.d. family of the given random string of beads. The following is a first recursive construction.

\begin{theorem}[Recursive construction of binary CRTs] \label{introthm}
Let $p\ge 1$, $\beta>1/p$ and let $\xi=(\check{\mathcal T}_0, \check{\mu}_0)$ be a random string of beads $\xi$ of length $L$ with $\mathbb{E}[L^p]<\infty$. For $n \geq 0$, to obtain $(\check{\mathcal T}_{n+1}, \check{\mu}_{n+1})$ conditionally given $(\check{\mathcal T}_{n}, \check{\mu}_{n})$, attach to each atom $x \in \check{\mathcal T}_{n}$ of $\check{\mu}_{n}$ an independent isometric copy of $\xi$ with metric rescaled by $\check{\mu}_n(x)^{\beta}$, and mass measure rescaled by $\check{\mu}_{n}(x)$. Then there exists a compact CRT $(\check{\mathcal T}, \check{\mu})$ such that
\begin{equation*} \lim \limits_{n \rightarrow \infty} \left(\check{\mathcal T}_n, \check{\mu}_n\right)= \left(\check{\mathcal T}, \check{\mu}\right) \quad \text{a.s. in the Gromov-Hausdorff-Prokhorov topology on $\mathbb T_{\rm w}$}. \end{equation*}
\end{theorem} 

The attachment procedure for the above construction will be defined precisely in Section \ref{Sec4}. Theorem \ref{introthm} implies the convergence to a CRT of the bead splitting processes as introduced in \cite{37}, based on an arbitrary random strings of beads. 

\begin{corollary}[Bead splitting processes] \label{introcornew}
  Let $p\ge 1$, $\beta>1/p$, and let $\xi=({\mathcal T}_0, {\mu}_0)$ be a random string of beads of length $L$ with $\mathbb{E}[L^p]<\infty$. For $k \geq 0$, to 
  obtain $({\mathcal T}_{k+1},{\mu}_{k+1})$ conditionally given $({\mathcal T}_{k}, \mu_{k})$,  pick an atom $J_k \in {\mathcal T}_{k}$ from ${\mu}_{k}$ and attach 
  at $J_k$ an independent isometric copy of $\xi$ with metric rescaled by ${\mu}_k(x)^{\beta}$, and mass measure rescaled by ${\mu}_{k}(x)$. Then there exists a compact
  CRT $({\mathcal T},\mu)$ such that
$$ \lim \limits_{k \rightarrow \infty} \left({\mathcal T}_k,\mu_k\right)= \left({\mathcal T},\mu\right) \quad \text{a.s. in the Gromov-Hausdorff-Prokhorov topology on $\mathbb T_{\rm w}$}.$$
\end{corollary}

We will prove this corollary by embedding $({\mathcal T}_k,\mu_k)$, $k\ge 0$, into $(\check{\mathcal T},\check{\mu})$, and then showing $({\mathcal T},\mu)=(\check{\mathcal T},\check{\mu})$. Corollary \ref{introcornew} was proved in \cite[Theorem 21]{37} in the special case when the string of beads has a regenerative property in the sense of Gnedin and Pitman \cite{3}. Then $({\mathcal T},\mu)$ is a self-similar CRT where a weighted $\mathbb R$-tree $(\mathcal T, d, \mu)$ is called \textit{self-similar} if for all $t \geq 0$, conditionally given the subtree masses $(\mu(\mathcal T_i(t)), i \geq 1)$ of the connected components $(\mathcal T_i(t), i \geq 1)$ of $\{x \in \mathcal T: d(\rho, x) > t\}$, the trees $(\mathcal T_i(t), i \geq 1)$ have the same distribution as independent isometric copies of $\mathcal T$ with metric rescaled by $\mu(\mathcal T_i(t))^\beta$ and mass measures by $\mu(\mathcal T_i(t))$, cf. Haas and Miermont \cite{33} and Stephenson \cite{34}. Since we show $({\mathcal T},\mu)=(\check{\mathcal T},\check{\mu})$, Theorem \ref{introthm} gives an alternative construction of binary self-similar CRTs of \cite{33}. Examples of self-similar CRTs include Aldous' Brownian CRT \cite{6,7,8}, and Duquesne and Le Gall's stable trees \cite{36,38,16,17,18} parametrised by some $\theta \in (1,2]$. The stable tree of index $\theta=2$ is the Brownian CRT, which is binary. 

General self-similar weighted $\mathbb R$-trees can have branch points of any finite or infinite degree (as is the case for stable trees), \cite{33}, continuous mass on branches, atoms on branches and in leaves, \cite{34}. To capture these features, we add more structure to strings of beads. Specifically, while a string of beads can be represented as $([0,\ell], \sum_{i \geq 1} p_i \delta_{x_i})$ for some $p_1 \geq p_2 \geq \cdots \geq 0$, $\sum_{i \geq 1} p_i=1$, $x_i \in [0,\ell]$ distinct, we also allow $x_i$ not necessarily distinct, a measure $\lambda$ on $[0,\ell]$ and $p_i \geq 0$ with $\sum_{i\geq 1}p_i=1-\lambda([0,\ell])$. We refer to $([0,\ell], (x_i)_{i \geq 1}, (p_i)_{i \geq 1}, \lambda)$ as a \textit{generalised string}. The map \eqref{AB1} is naturally defined for generalised strings, leaving
mass $\lambda([0,\ell])$ on the branch $[0,\ell]$ according to $\lambda$. Applying $\phi_\beta$ to random generalised strings $\xi=([0,L],(X_i)_{i\geq 1},(P_i)_{i\geq 1},\Lambda)$ also leads to a more general version of Theorem \ref{introthm}.

\begin{theorem} \label{constr2} Let $p\ge 1$, let $\xi=(\check{\mathcal T}_0, (\check{X}_i^{(0)})_{i \geq 1},  (\check{P}_i^{(0)})_{i \geq 1},\check{\Lambda}_0)$ be a random generalised string of length $L$ with $\mathbb{E}[L^p]<\infty$, and let $\beta\in(0,\infty)$ such that $\mathbb{E}[\sum_{j\ge 1}P_j^{p\beta}]<1$. For $n \geq 0$, to obtain $(\check{\mathcal T}_{n+1}, (\check{X}_i^{(n+1)})_{i \geq 1},  (\check{P}_i^{(n+1)})_{i \geq 1},\check{\Lambda}_{n+1})$ conditionally given $(\check{\mathcal T}_{n}, (\check{X}_i^{(n)})_{i \geq 1},  (\check{P}_i^{(n)})_{i \geq 1},\check{\Lambda}_{n})$, attach to each $\check{X}_i^{(n)} \in \check{\mathcal T}_n$ an independent isometric copy of $\xi$ with metric rescaled by $(\check{P}^{(n)}_i)^\beta$ and measure/atom masses rescaled by $\check{P}^{(n)}_i$. Let $\check{\mu}_n=\check{\Lambda}_n + \sum_{i \geq 1}\check{P}_i^{(n)} \delta_{\check{X}_i^{(n)}}$. Then there is a random weighted $\mathbb R$-tree $(\check{\mathcal T}, \check{\mu})$ such that
\begin{equation*} \lim \limits_{n \rightarrow \infty} \left(\check{\mathcal T}_n, \check{\mu}_n\right)= \left(\check{\mathcal T}, \check{\mu}\right) \quad \text{a.s. in the Gromov-Hausdorff-Prokhorov topology on $\mathbb T_{\rm w}$}. \end{equation*}
\end{theorem}

In the above construction, Theorem \ref{constr2}, we obtain a CRT (in the strict sense defined above) if and only if $\xi$ has $\Lambda=0$ and $L=\sup_{i\ge 1\colon P_i>0}X_i$. By our method, every self-similar CRT (indeed every random weighted $\mathbb R$-tree constructed in Theorem \ref{constr2}) is uniquely characterised by fixpoint equations. As an example, we obtain a new fixpoint characterisation of the stable trees, in the case where $\xi$ is a $\beta$-generalised string, which we define as follows, in terms of Poisson-Dirichlet distributions
\cite{5}.

\begin{definition}[$(\beta, \beta)$-string of beads and $\beta$-generalised string]\label{introdef}
Let $\beta \in (0,1)$, $(Q_m, m \geq 1) \sim {\rm PD}(\beta, \beta)$ with $\beta$-diversity $L=\lim_{m \rightarrow \infty} m \Gamma(1-\beta) Q_m^\beta$, and $(U_m, m\geq 1)$ i.i.d. Unif$([0,1])$. The weighted random interval $([0,L], \sum_{m \geq 1} Q_m \delta_{L U_m})$ is called a \textit{$(\beta, \beta)$-string of beads} \cite{1}. 

For $\beta \in (0,1/2]$, consider $(R_j^{(m)}, j \geq 1) \sim {\rm PD}(1-\beta, -\beta)$, $m\ge 1$, i.i.d. and let $(P_i, i \geq 1)$ be the decreasing rearrangement of $(Q_m R_{j}^{(m)}, j \geq 1, m \geq 1)$ and $X_i=LU_m$ if $P_i=Q_mR_j^{(m)}$. For $\lambda=0$, $([0,L], (X_i)_{i \geq 1}, (P_i)_{i \geq 1}, \lambda)$ is called a \textit{$\beta$-generalised string}. \end{definition}

\begin{theorem} \label{fixp} The distribution of the stable tree of index $\theta \in (1,2]$ is the unique solution to the distributional fixpoint equation \eqref{AB1} associated with a $(1-1/\theta)$-generalised string. The fixpoint is attractive. 
\end{theorem}

In particular, this establishes that the Brownian CRT is the unique attractive fixpoint of \eqref{AB1} in the case of a $(1/2,1/2)$-string of beads $\xi$. See also recent work by Albenque and Goldschmidt \cite{50}, who proved that the Brownian CRT is the unique attractive fixpoint of a different distribution equation obtained by joining three i.i.d. weighted $\mathbb R$-trees at randomly chosen vertices (sampled from the respective mass measure), scaled by the parts of an independent Dirichlet($1/2, 1/2, 1/2)$ split. 

If we sacrifice the limiting weight measure $\check{\mu}$ on $\check{\mathcal T}$, we can obtain the existence of a unique distributional fixpoint of \eqref{AB1} and the convergence of the trees $\check{\mathcal T}_n$ constructed as in Theorem \ref{constr2} for yet more general $\xi$, where the random interval $[0,L]$ equipped with a sequence of masses
$P_i$, $i\geq 1$, in not necessarily distinct locations $X_i\in[0,L]$, may have $\sum P_i>1$, even $\sum P_i=\infty$, as long as the $P_i$ decrease \enquote{fast enough}. We could, of course, include a measure $\Lambda$ on $[0,L]$, but its only purpose in Theorem \ref{constr2} was to provide mass on branches for $\check{\mu}$, and $\check{\mu}$ will no longer exist in 
this generality. Let us now state this our most general fixpoint theorem, which holds in the subspace $\mathcal P_p\subset\mathcal P(\mathbb T)$ of distributions of random trees whose height 
${\rm ht}(\mathcal T)=\sup_{x\in\mathcal T}d(\rho,x)$ has finite $p$th moment. We equip $\mathcal P_p$ with the Wasserstein distance $W_p$. See the end of Section \ref{rtp} for details.

\begin{theorem}[Fixpoint] \label{uniquefix} Let $\beta \in (0,\infty)$, $p\ge 1$, and let $\xi=([0,L], (X_i)_{i \geq 1}, (P_i)_{i \geq 1})$ be such that $0\le X_i\le L$, $P_i\ge 0$, $i\ge 1$, $\mathbb E[L^p]<\infty$ and $\mathbb E[\sum_{j \geq 1}P_j^{p\beta}] < 1$. Then the distributional equation \eqref{AB1} associated with $\xi$ has a unique attractive fixpoint in 
  $(\mathcal P_p,W_p)$.
\end{theorem}

The following is a corresponding recursive construction of the fixpoint, as a Gromov-Hausdorff limit.

\begin{theorem}[Recursive construction]\label{constr3} In the setting of Theorem \ref{uniquefix}, using notation of Theorem \ref{constr2}, 
  $$\lim \limits_{n \rightarrow \infty} \check{\mathcal T}_n = \check{\mathcal T} \text{ a.s. in the Gromov-Hausdorff topology on $\mathbb T$}$$ 
  for some random compact $\mathbb R$-tree $\check{\mathcal T}$.
\end{theorem}

In this introduction we focussed on the use of fixpoint equations and recursive constructions to obtain large classes of CRTs, some of which are well-known in other
contexts, but many of which are new, and we mentioned related work on tree growth processes and fixpoint characterisations. \smallskip

\noindent \bf Applications and examples \rm of our results include the following, which we present in Section \ref{appl}. 
\begin{itemize} 
  \item We demonstrate how our constructions include new constructions of the self-similar trees of \cite{33,34}, whose existence was established there 
    using different methods. 
  \item We give, for the first time, constructions of the genealogical trees associated with Bertoin's self-similar growth fragmentations, including those related to the Brownian map \cite{54,55}. 
  \item Our methods establish moments for the height of the fixpoint tree, which corresponds to the extinction time in the context of growth fragmentations. In \cite{54},
    moment results were obtained only in the spectrally positive case. Our methods work more generally and notably include a one-parameter 
    class studied in \cite{54} as an extension of the growth fragmentation relating to the Brownian map.
  \item We obtain general Hausdorff dimension results for trees $\check{\mathcal{T}}$ of Theorem \ref{constr3}. 
  \item We construct a specific binary CRT, which we apply in forthcoming work \cite{forth} as an example of an embedding problem for CRTs, providing a binary embedding
    of the stable line-breaking construction, which will solve an open problem of Goldschmidt and Haas \cite{12}.
\end{itemize}
%We leave open 
%\begin{itemize} 
%  \item the construction of an intrinsic mass measure on the genealogies of growth fragmentations,
%  \item the use of doubly marked tilted trees \cite{34} for Hausdorff dimension calculations including for Theorem \ref{constr3}.	
%\end{itemize}

This article is organised as follows. In Section \ref{Sec3} we give an introduction to recursive distribution equations, recursive tree frameworks and $\mathbb R$-trees. Section \ref{Sec4} turns to the existence of random $\mathbb R$-trees as distributional fixpoints and their recursive constructions, including the proofs of most of the results presented in this introduction. In Section \ref{appl} we present examples and applications of our method.%, in particular we obtain the genealogical trees of Bertoin's self-similar growth fragmentations.

\section{Preliminaries} \label{Sec3}

\subsection{Recursive distribution equations and recursive tree frameworks}

We briefly review the concept of a \textit{recursive tree framework} (RTF) as presented by Aldous and Bandyopadhyay \cite{14}, where general recursive distributional equations were studied with regard to the existence of fixpoints. Such recursive relationships arise in a variety of contexts, e.g. in algorithmic structures, Galton-Watson branching processes and combinatorial random tree structures. We will use the recursive distributional relations underlying an RTF to give a recursive construction of CRTs based on random strings of beads. While our notation is suggestive, the generality of \cite{14} is as stated here.

Consider two measurable spaces $(\mathbb T,  \mathcal A_{\mathbb T})$ and $(\Xi, \mathcal A_{\Xi})$, and the product space
\begin{equation} \Xi^*:=\Xi \times \bigcup_{0 \leq m \leq \infty} \mathbb T^m, \end{equation}
where $\mathbb T^m$ denotes the space of $\mathbb T$-valued sequences of length $m$, $0 \leq  m \leq \infty$, including $\bT^0:=\{\Upsilon\}$, where $\Upsilon$ stands
for the empty sequence. Furthermore, consider a measurable map \begin{equation} \phi: \Xi^* \rightarrow  \mathbb T, \label{map1} \end{equation}  and random variables $(\xi, N) \in \Xi \times \overline{\mathbb N}:=\Xi \times \{1, 2, \ldots; \infty\}$, and $(\tau_i, i \geq 1) \in \mathbb T^\infty$ as follows.
\begin{itemize}
\item[(i)] The pair $(\xi, N)$ has some distribution $\nu$, i.e. $(\xi, N) \sim \nu$.
\item[(ii)] The sequence $(\tau_i, i \geq 1)$ is i.i.d. with some distribution $\eta$, i.e. $\tau_i \sim \eta$, $i\ge 1$.
\item[(iii)] The random variables in (i) and (ii) are independent.
\end{itemize}

We denote by $\mathcal P(\mathbb T)$ the set of probability measures on the space $(\mathbb T,\mathcal A_{\mathbb T})$. For any given distribution $\nu$ on $\Xi \times \overline{\mathbb N}$, we then obtain a map
\begin{equation} \Phi: \mathcal P(\mathbb T) \rightarrow \mathcal P(\mathbb T), \quad \eta \mapsto \Phi(\eta), \label{phi1} \end{equation}
where $\Phi(\eta)$ is defined as the distribution of
\begin{equation} \tau:=\phi \left(\xi, \tau_i, 1 \leq i \leq N\right). \label{relranvar} \end{equation}
We call $\tau$ the \textit{parent} value of $(\tau_i, 1 \leq i \leq N)$, and refer to $(\tau_i, 1 \leq i \leq N)$, as the values of the \textit{children} of $\tau$. We now view the random variables $\tau_i$ as the parent values of random variables $\tau_{ij}, 1 \leq j \leq N_i$, associated with $\xi_i$, i.e. 
\begin{equation*}
\tau_{i}=\phi\left(\xi_i, \tau_{ij}, 1 \leq j \leq N_i\right), \quad 1 \leq i \leq N.
\end{equation*}
This setting can be extended recursively to each of the following generations, i.e. each child is considered as a parent itself. We refer to this setting as a \textit{recursive tree framework}. More precisely, we define 
\begin{equation*} \mathbb U:=\bigcup_{n \geq 0} \mathbb N^n \end{equation*} 
using the Ulam-Harris notation to describe the set of all possible descendants $\mathbf i$, where $\mathbf i=i_1 i_2\cdots i_n\in\mathbb N^n$ denotes an individual in    generation $n \geq 1$, which is the $i_n$-th child of the parent $i_1i_2 \cdots i_{n-1}\in\mathbb N^{n-1}$. Note that in this terminology $\left\{\tau_{\textbf i}: \text{{\textbf i} in generation } n\right\}=\left\{\tau_{\textbf i}: {\textbf i} \in \mathbb N^n \right\}$. If we let the empty vector $\varnothing$ (the only element of $\mathbb N^0$) be the root of $\mathbb U$ and link parents and children via edges, the set $\mathbb U$ can be viewed as an infinite (discrete) tree. 

\begin{definition}[Recursive tree framework]  A pair $((\xi_{\textbf i}, N_{\textbf i}, \textbf i \in \mathbb U)   ; \phi)$ where $(\xi_{\textbf i}, N_{\textbf i}, \textbf i \in \mathbb U)$ is a sequence of i.i.d. $\Xi \times \overline{\mathbb N}$-valued random variables $(\xi_{\textbf i}, N_{\textbf i}) \sim \nu$, $\textbf i \in \mathbb U$, and $\phi:  \Xi^* \rightarrow \mathbb T$ is a measurable map is called a \textit{recursive tree framework} (RTF).  \end{definition}

Suppose there are random variables $\tau_{\textbf i}, {\textbf i} \in \mathbb U$, possibly on an extended probability space, as follows.
\begin{itemize}
\item[(i)] For all $\textbf{i} \in \mathbb U$, 
\begin{equation} \tau_{\textbf i}=\phi\left(\xi_{\textbf i}, \tau_{{\textbf i}j}, 1 \leq j \leq N_{\textbf i}\right) \quad \text{ a.s..} \label{45} \end{equation}
\item[(ii)] The random variables \begin{equation} \left\{\tau_{\textbf i}: {\text{{\textbf i} in generation }} n\right\} \label{46} \end{equation} are i.i.d. with some distribution $\eta_n$ on $(\mathbb T,\mathcal A_{\mathbb T})$.
\item[(iii)] The random variables $\left\{\tau_{\textbf i}: {\text{{\textbf i} in generation } n}\right\}$ are independent of the random variables \begin{equation}  \left\{(\xi_{\textbf i}, N_{\textbf i}): \text{{\textbf i} in generations $0,\ldots,n-1$}\right\}. \end{equation}. \end{itemize} \vspace{-0.6cm}

A \textit{recursive tree process} (RTP) is a recursive tree framework with random variables $\tau_{\textbf i}, \textbf{i} \in \mathbb U$, as in \eqref{45}-\eqref{46}, i.e. an RTP is an RTF enriched by the random variables $\tau_{\textbf i}, \textbf{i} \in \mathbb U$. An RTP with random variables $\tau_{\textbf i}$ only defined up to generation $n$, i.e. only for ${\textbf i} \in \bigcup_{m=0}^n \mathbb N^m$, is called an RTP \textit{of depth n}. While RTPs of depth $n$ can always be defined using \eqref{46} for any distribution $\eta_n$ and $\eqref{45}$ for generations $n-1, \ldots, 0$, RTPs of infinite depth do not exist in general. We refer to \cite[Section 2.3]{14} for more details on RTFs and RTPs, in particular with regard to connections to Markov chains and Markov transition kernels. 

In what follows, consider a fixed recursive tree framework, i.e. let $(\xi_{\textbf i}, N_{\textbf i})$, $\textbf i \in \mathbb U$, be i.i.d. with distribution $\nu$, and let $\phi: \Xi^* \rightarrow \mathbb T$ be a measurable map. Given $n \geq 1$ and an arbitrary distribution $\eta_n$ on $\mathbb T$, we consider a recursive tree process where the values of $n$-th generation individuals are i.i.d. with distribution $\eta_n$. The distributions of the values of $j$-th generation individuals are then given by
\begin{equation} \tau_{\textbf i} \sim \eta_j:=\Phi^{n-j}(\eta_n), \quad \textbf{i} \in \mathbb N^{j}, \end{equation}
for $1 \leq j \leq n-1$. Aldous and Bandyopadhyay \cite{14} studied fixpoints of $\Phi$. Note that the existence of a fixpoint $\eta^*$ of $\Phi$ ensures the existence of a stationary RTP, i.e. an RTP with $\eta_n=\eta$, $n \geq 0$, by Kolmogorov's consistency theorem.

\begin{lemma}[The contraction method, \cite{14} Lemma 5] \label{contrac}
Let $\mathcal P \subset \mathcal P(\mathbb T)$ such that $\Phi(\mathcal P) \subset \mathcal P$, i.e. consider $\Phi \colon \mathcal P \rightarrow \mathcal P$ as in \eqref{phi1}-\eqref{relranvar} related to a recursive tree process as above. Furthermore, let $d_{\mathcal P}$ be a complete metric on $\mathcal P$ such that the contraction property holds, i.e.
\begin{equation} \sup_{\eta, \eta' \in \mathcal P, \eta \neq \eta'} \frac{d_{\mathcal P}\left(\Phi(\eta), \Phi(\eta')\right)}{d_{\mathcal P}\left(\eta, \eta'\right)}< 1. \end{equation}
Then the map $\Phi$ has a unique fixpoint $\eta^* \in \mathcal P$, and the domain of attraction of $\eta^*$ is the whole set $\mathcal P$, i.e. for all $\eta\in\mathcal{P}$, we have $\displaystyle\lim_{j\rightarrow\infty}\Phi^j(\eta)=\eta^*$ in $(\mathcal P,d_{\mathcal P})$.
\end{lemma}

\subsection{Weighted $\mathbb R$-trees and the Gromov-Hausdorff-Prokhorov topology} \label{Sec31}

We use the notion of an {$\mathbb R$-tree}, that is, a separable metric space $(T, d)$ such that the following two properties hold for every $\sigma_1, \sigma_2 \in T$.
\begin{enumerate}
\item[(i)] There is an isometric map $h_{\sigma_1, \sigma_2} \colon [0, d(\sigma_1, \sigma_2)] \rightarrow T$ such that \begin{equation*} h_{\sigma_1, \sigma_2}(0)=\sigma_1 \text{ and } h_{\sigma_1, \sigma_2}(d(\sigma_1, \sigma_2))=\sigma_2. \end{equation*}
\item[(ii)] For every injective path $q \colon [0,1] \rightarrow T$ with $q(0)=\sigma_1$ and $q(1)=\sigma_2$ we have\begin{equation*} q([0,1])=h_{\sigma_1, \sigma_2}([0,d(\sigma_1, \sigma_2)]). \end{equation*}
\end{enumerate}
We write $[[\sigma_1, \sigma_2]]:=h_{\sigma_1, \sigma_2}\left([0,d(\sigma_1, \sigma_2\right)])$ for the range of $h_{\sigma_1, \sigma_2}$. The trees considered in this paper are usually compact, but we also allow non-compact $\mathbb R$-trees. A \textit{rooted} $\mathbb R$-tree $(T, d, \rho)$ is an $\mathbb{R}$-tree $(T, d)$ with a distinguished element $\rho\in T$, the \textit{root}. We only consider rooted $\mathbb{R}$-trees, and will often refer to $T$ as an $\mathbb R$-tree without mentioning the distance $d$ and the root $\rho$ explicitly. For any $c > 0$ and any metric space $(T, d)$, we write $c T$ for $(T, c d)$, the metric space obtained when all distances are multiplied by $c$.

We are only interested in equivalence classes of rooted $\mathbb R$-trees. Two rooted $\mathbb R$-trees $(T, d, \rho)$ and $(T', d', \rho')$ are \textit{equivalent} if there exists an isometry from $T$ onto $T'$ such that $\rho$ is mapped to $\rho'$.  The set of equivalence classes of compact rooted $\mathbb R$-trees is denoted by $\mathbb T$.

We follow \cite{21} and equip $\mathbb T$ with the (pointed) \textit{Gromov-Hausdorff distance} $d_{\rm GH}$. For rooted $\mathbb{R}$-trees $(T, d, \rho)$, $(T',d', \rho')$
we define  \begin{equation}  d_{\rm GH} \left(\left(T, d,\rho\right), \left(T', d',\rho'\right)\right) := \inf \limits_{\varphi, \varphi'} \left\{\max \left\{\delta \left(\varphi\left(\rho\right), \varphi'\left(\rho'\right)\right), \delta_{\rm H}\left(\varphi \left(T\right), \varphi'\left(T'\right)\right)\right\}\right\}, \label{GHdist} \end{equation}
where $\delta_{\rm H}$ is the Hausdorff distance between compact subsets of $(\mathcal M, \delta)$, and the infimum is taken over all metric spaces $(\mathcal M, \delta)$ and all isometric embeddings $\varphi\colon T \rightarrow \mathcal M$, $\varphi' \colon T' \rightarrow \mathcal M$ into $(\mathcal M, \delta)$. The Gromov-Hausdorff distance only depends on the equivalence classes of $(T, d, \rho)$, $(T',d', \rho')$ and induces a metric on $\bT$, which we also denote by $d_{\rm GH}$. We equip $\mathbb T$ with the associated Borel $\sigma$-algebra $\mathcal B(\mathbb T)$.

A \textit{weighted} $\mathbb R$-tree $(T, d, \rho, \mu)$ is a rooted $\mathbb R$-tree $(T, d, \rho)$ equipped with a probability measure $\mu$ on the Borel sets $\mathcal B(T)$ of $(T,d)$. Two {weighted $\mathbb R$-trees} $(T, d, \rho, \mu)$ and $(T', d', \rho', \mu')$ are \textit{equivalent} if there is an isometry from $(T, d, \rho)$ onto $(T', d', \rho')$ such that $\mu'$ is the push-forward of $\mu$ under this isometry. The set of equivalence classes of weighted compact $\mathbb R$-trees is denoted by $\mathbb{T}_{\rm w}$.

The Gromov-Hausdorff distance can be extended to a distance between weighted $\mathbb R$-trees, the \textit{Gromov-Hausdorff-Prokhorov distance} between two weighted $\mathbb{R}$-trees $(T, d, \rho, \mu)$ and $(T',d', \rho', \mu')$, 
\begin{equation} d_{\rm GHP} \left(\left(T, d,\rho, \mu\right), \left(T', d',\rho', \mu'\right)\right) := \inf \limits_{\varphi, \varphi'} \left\{\max\left\{\delta\left(\varphi\left(\rho\right), \varphi'\left(\rho'\right)\right), \delta_{\rm H}\left(\varphi\left(T\right), \varphi'\left(T'\right)\right), \delta_{\text P}\left(\varphi_*\mu, \varphi'_*\mu'\right) \right\} \right\}, \label{GHPdist} \end{equation}
where $(\mathcal M,\delta), \varphi, \varphi', \delta_{\text H}$ are as in \eqref{GHdist}, $\varphi_*\mu$, $\varphi'_*\mu$ are the push-forwards of $\mu$, $\mu'$ via $\varphi, \varphi'$, respectively, and $\delta_{\text P}$ is the Prokhorov distance on the space of Borel probability measures on $(\mathcal M, \delta)$ given by
\begin{equation*}
\delta_{\rm P}\left(\mu, \mu' \right)=\inf \left\{ \epsilon >0: \mu(D) \leq \mu'( D^\epsilon)+\epsilon \quad \forall D \subset \mathcal M \text{ closed}\right\}
\end{equation*}
where $D^\epsilon:=\{x \in \mathcal M: \inf_{y \in D} \delta(x,y) <\epsilon\}$ is the \textit{$\epsilon$-thickening} of $D\subset\mathcal M$. The Gromov-Hausdorff-Prokhorov distance only depends on the equivalence classes and induces a metric on $\mathbb T_{\rm w}$. 

\begin{prop}[e.g. \cite{20,22,HM12}]\label{sepcompt} The spaces $(\mathbb T,d_{\rm GH})$ and $(\mathbb T_{\rm w}, d_{\rm GHP})$ are separable and complete. \end{prop}

We will also need some terminology to describe an $\mathbb R$-tree $(T,d, \rho)$.
For any $x \in T$, we call $d(\rho, x)$ the \textit{height of} $x$, ${\rm ht}(T):=\sup_{x \in T}d(\rho,x)$ the \textit{height of} $T$.  A \textit{leaf} is an element $x \in T\setminus \{\rho\}$ such that $T \setminus \{x\}$ is connected and we denote the set of all leaves of $T$ by ${\rm Lf}(T)$. An element $x \in T\setminus \{\rho\}$ is a \textit{branch point} if $T \setminus \{x\}$ has at least three connected components. The \textit{degree} ${{{\rm deg}}}(x, T)$ of a vertex $x \in T$ is the number of connected components of $T \setminus \{x\}$. 
%The connected components of $T$ obtained when we remove all leaves and all branch points of $T$ are called $\textit{edges}$. 

Following Aldous \cite{8,7,6}, we call a weighted $\mathbb R$-tree $(T, d, \rho, \mu)$ a \textit{continuum tree} if the probability measure $\mu$ on $T$ satisfies the following three additional properties.
\begin{enumerate}
\item[(i)] $\mu$ is supported by ${\rm Lf} (T)$, the set of leaves of $T$.
\item[(ii)] $\mu$ has no atom, i.e. for any singleton $x \in {\rm Lf}(T)$ we have $\mu(x):=\mu(\{x\})=0$. 
\item[(iii)] For every $x \in T \setminus {\rm Lf}(T)$, $\mu(T_x)>0$, where $T_x := \{ \sigma \in T: x \in [[\rho, \sigma]]\}$ is the subtree above $x$ in $T$.
\end{enumerate} 

Note that these conditions imply that ${\rm Lf} (T)$ is uncountable and has no isolated points. We refer to \cite{20,25,26} for more details on the topic of $\mathbb R$-trees. 

While some of our developments are more easily stated and/or proved in $(\mathbb T,d_{\rm GH})$ or $(\mathbb T_{\rm w},d_{\rm GHP})$, others benefit from more explicit
embeddings into a particular metric space $(\mathcal{M},\delta)$, which we will always choose as 
$$\mathcal{M}=l^1(\mathbb U):=\left\{(s_{\mathbf i})_{\mathbf{i}\in\mathbb U}\in[0,\infty)^{\mathbb U}\colon\sum_{\mathbf{i}\in\mathbb U}s_{\mathbf i}<\infty\right\}$$
equipped with the metric induced by the $l^1$-norm. Since $\mathbb U$ is countable, this is only a very slight variation of Aldous's \cite{8,7,6} choice $\mathcal{M}=l^1(\mathbb{N})$.
We denote by $\mathbb{T}^{\rm emb}$ the space of all compact $\mathbb{R}$-trees $T\subset l^1(\mathbb{U})$ with root $0\in T$, which we equip with the Hausdorff metric $\delta_{\rm H}$, and by $\mathbb{T}^{\rm emb}_{\rm w}$ the space of all weighted compact $\mathbb{R}$-trees $(T,\mu)$ with $T\in\mathbb{T}^{\rm emb}$, which we equip with the metric
$\delta_{\rm HP}((T,\mu),(T^\prime,\mu^\prime))=\max\{\delta_{\rm H}(T,T^\prime),\delta_{\rm P}(\mu,\mu^\prime)\}$.

\begin{prop}\label{embprop}\begin{enumerate}\item[\rm(i)] $(\mathbb{T}^{\rm emb},\delta_{\rm H})$ and $(\mathbb{T}^{\rm emb}_{\rm w},\delta_{\rm HP})$ are separable and complete.
  \item[\rm(ii)] For all $T,T^\prime\in\mathbb{T}^{\rm emb}$ we have $d_{\rm GH}(T,T^\prime)\le\delta_{\rm H}(T,T^\prime)$, and for all $(T,\mu),(T^\prime,\mu^\prime)\in\mathbb{T}^{\rm emb}_{\rm w}$, we have $d_{\rm GHP}((T,\mu),(T^\prime,\mu^\prime))\le\delta_{\rm HP}((T,\mu),(T^\prime,\mu^\prime))$. 
  \item[\rm(iii)] Every rooted compact $\mathbb{R}$-tree is equivalent to an element of $\mathbb{T}^{\rm emb}$, and every rooted weighted compact $\mathbb{R}$-tree is
    equivalent to an element of $\mathbb{T}^{\rm emb}$.
  \end{enumerate}  
\end{prop}  
\begin{proof} This is well-known. (ii) is trivial. The remainder is easily deduced from known properties of Hausdorff and Prokhorov metrics. See for example \cite[Propositions 3.6 and 3.7]{DuWi1} for the statements of (i) and (iii) in the case of  $(\mathbb{T}^{\rm emb},\delta_{\rm H})$. 
\end{proof}

\section{Construction of CRTs using recursive tree processes} \label{Sec4}

\subsection{The setting for recursive tree frameworks relating $\mathbb R$-trees and generalised strings} \label{rtp}
\label{setting}
We now use a specific recursive tree framework to construct (possibly weighted) random $\mathbb R$-trees out of i.i.d. copies of a random string of beads or a random generalised string $\xi=([0,L],(X_i)_{i\ge 1},(P_i)_{i\ge 1},\Lambda)$.

Let $\mathbb T$ be the space of equivalence classes of rooted compact $\mathbb R$-trees, as in Section \ref{Sec31}, and let $\widetilde{\Xi}_{\rm s}$ be the set of strings of beads defined by
\begin{equation}
\widetilde{\Xi}_{\rm s} = \left\{\left([0,\ell],\sum_{i \geq 1}p_i \delta_{x_i} \right) \text{ such that (i)$_{\rm s}$, (ii)$_{\rm s}$, (iii) hold} \right\}  \label{Xi} \end{equation} 
where the properties (i)$_{\rm s}$, (ii)$_{\rm s}$, (iii) are given by
\begin{itemize}
\item[(i)$_{\rm s}$] $\ell>0$ and $x_i\in[0,\ell]$, $i\ge 1$, distinct, with $\ell=\sup\{x_i\colon p_i>0,i\ge 1\}$;
\item[(ii)$_{\rm s}$] $1 > p_1\geq p_2\geq\cdots\geq 0$ with $\sum_{i\ge 1}p_i=1$;
\item[(iii)$_{\,\,}$] the sequence $(x_i)_{i \geq 1}$ has indices assigned in decreasing order of the masses $(p_i)_{i \geq 1}$; indices are assigned according to increasing distance to $0$ if atom masses have the same size.
% {\tt what if $x_i=x_j$ and $P_i=P_j$?}.
\end{itemize}
Each element in $\widetilde{\Xi}_{\rm s}$ is characterised via a constant $\ell>0$ and a sequence of distinct atoms $(x_i)_{i \geq 1}$ with respective masses $(p_i)_{i \geq 1}$ in decreasing order, summing to 1. Therefore, we consider the set 
\begin{equation*}\Xi_{\rm s}:=\left\{\left([0,\ell], \left(x_i\right)_{i \geq 1}, \left(p_i\right)_{i \geq 1}\right)\text{ such that (i)$_{\rm s}$, (ii)$_{\rm s}$, (iii) hold}\right\}\end{equation*} 
instead of $\widetilde{\Xi}_s$, noting that $\Xi_{\rm s}$ and $\widetilde{\Xi}_{\rm s}$ are in natural one-to-one correspondence if we enforce a further convention about $x_i$ when $p_i=0$, which it is sometimes more convenient not to do.%Note that the $x_i's$ are not necessarily distinct.

We will also consider the set $\Xi_{\rm g}$ of generalised strings given by  
\begin{equation*}
{\Xi}_{\rm g} = \left\{ \left([0,\ell], (x_i)_{i \geq 1)}, (p_i)_{i \geq 1}, \lambda \right)\text{ where } \lambda \text{ is a measure on } [0,\ell] \text{ such that (i), (ii)$_{\rm g}$, (iii) hold} \right\} \end{equation*} where 
\begin{itemize}
\item[(i)$_{\,\;}$] $\ell>0$ and $x_i\in[0,\ell]$, $i\ge 1$, not necessarily distinct;
\item[(ii)$_{\rm g}$] $1 > p_1\geq p_2\geq\cdots\geq 0$ with $\sum_{i\ge 1}p_i=1-\lambda\left([0,\ell]\right)$.
\end{itemize}
Finally, we allow not necessarily summable atom masses $(p_i)_{i \geq 1}$ in the space $\Xi$ given by 
\begin{equation*}
{\Xi} = \left\{ \left([0,\ell], (x_i)_{i \geq 1)}, (p_i)_{i \geq 1} \right)\text{ such that (i), (ii), (iii) hold} \right\}  \subset [0,\infty) \times [0,\infty)^{\mathbb N} \times l^\infty(\mathbb N) 
\end{equation*} 
equipped with the subset topology of the natural product topology, where 
\begin{itemize}
\item[(ii)] $p_1\geq p_2\geq\cdots\geq 0$.\pagebreak[2]
\end{itemize}
We set $\Xi^*:= \Xi\times \mathbb T^{\infty}$ where $\mathbb T^{\infty}$ is the set of infinite sequences in $\mathbb T$.
We will work with the space $\Xi$ in Section \ref{fixsec} to establish the general fixpoint result of Theorem \ref{uniquefix} and the recursive construction of Theorem \ref{constr3}, which also yield all random $\mathbb{R}$-trees needed for the other theorems of the introduction. The spaces $\Xi_{\rm g}\supset\Xi_{\rm s}$ capture the
more restrictive settings of those other theorems and are used in Section \ref{resultsstrings} to add mass measures to the constructions.  
From now on, let $\beta \in (0, \infty)$ be fixed. We equip  $\Xi^*$ with the metric $d_{\beta}$ where $d_{\beta}((\xi,\tau_i,i\ge 1),(\xi',\tau_i',i\ge 1))$ is defined as 
\begin{equation*} \left|\ell - \ell' \right| \vee \sup_{i \geq 1}\left(\left| x_i - x_i'\right|  \vee \left| (p_i)^{\beta} - (p_i')^{\beta}\right|\vee d_{\rm GH}(\tau_i, \tau_i') \vee d_{\rm GH} \left((p_i)^{\beta}\tau_i, (p_i')^{\beta} \tau_i'\right)\right)  \end{equation*}
for $\xi=\left([0,\ell], (x_i)_{i \geq 1}, (p_i)_{i \geq 1}\right)$, $\xi'= \left([0,\ell'], (x'_i)_{i \geq 1}, (p'_i)_{i \geq 1}\right) \in \Xi$, where we recall that $p_i^\beta \tau_i$ denotes the $\mathbb R$-tree $\tau_i$ with distances rescaled by $p_i^\beta$.

\begin{prop} \label{sep} The metric space $\left(\Xi^*,d_{\beta}\right)$ is separable.\end{prop}

\begin{proof} We first show that the space $([0,\infty)\times \mathbb T, \tilde{d}_\beta)$ is separable, where  
\begin{equation}
\tilde{d}_\beta\left(\left(p,\tau\right), \left(p',\tau' \right)\right) :=\left|p^\beta-p'^\beta \right| \vee d_{\rm GH}\left(\tau, \tau'\right)\vee d_{\rm GH}\left(p^{\beta}\tau, \left(p'\right)^\beta\tau'\right)  .
\end{equation}
Recall that $\mathbb T$ equipped with the Gromov-Hausdorff distance $d_{\rm GH}$ and the space $[0,\infty)$ equipped with the Euclidian distance are separable, i.e. there exist countable dense subsets $\mathbb T' \subset \mathbb T$, $\mathbb Q \cap [0,\infty)\subset [0,\infty)$ such that for any $\epsilon >0$ and any $p \in [0,\infty)$, $\tau \in \mathbb T$ with ${\rm ht}(\tau)>0$, there are $p' \in \mathbb Q \cap (0,\infty)$, ${\tau'} \in {\mathbb T}'$ such that 
\begin{equation*} \left|p^\beta-\left(p'\right)^\beta \right| < \min\left\{\epsilon,{\epsilon}/\left(2 {\rm ht}\left(\tau\right)\right)\right\}, \qquad d_{\rm GH}\left(\tau, {\tau'} \right) < \min\left\{\epsilon,{\epsilon}/{2(p')^\beta}\right\},
\end{equation*}
see Proposition \ref{sepcompt}. Clearly, the set $D:=(\mathbb Q \cap[0,\infty))\times \mathbb T'$ is countable. Also,
\begin{align*}
d_{\rm GH}\left(p^{\beta}\tau, \left(p'\right)^\beta\tau'\right) &\leq d_{\rm GH}\left(p^{\beta}\tau, \left(p'\right)^\beta\tau\right) + d_{\rm GH}\left(\left(p'\right)^{\beta}\tau, \left(p'\right)^\beta\tau'\right) \\
& \leq \left|p^\beta-\left(p'\right)^\beta\right| {{\rm ht}}\left(\tau\right)+\left(p'\right)^\beta d_{\rm GH}\left(\tau, \tau'\right)\\
&< \left(\epsilon/\left(2{\rm ht}\left(\tau \right) \right) \right) {\rm ht}\left(\tau\right)+(p')^\beta{\epsilon}/{2(p')^\beta}=\epsilon,
\end{align*}
where we applied the triangle inequality. Hence, $D$ is dense in $[0,\infty)\times \mathbb T$, and hence $([0,\infty)\times \mathbb T, \tilde{d}_\beta)$ separable. Note that $[0,\infty)$ equipped with the Euclidian distance is separable. We complete the proof by noting the fact that if $(A_1, d_{1}), (A_2, d_{2}), \ldots$ are separable metric spaces, then the space $((A_1 \times A_2 \times \ldots), d_{\infty})$ is separable, where $d_{\infty}$ is defined by $d_{\infty}((a_1, a_2, \ldots),(a'_1, a'_2, \ldots)):=\sup_{j \geq 1} d_j(a_j, a_j')$.\end{proof}

Next, we consider $\xi=([0,\ell],(x_i)_{i\geq 1},(p_i)_{i\geq 1})\in\Xi$ and $(\tau_i, d_i, \rho_i)_{i \geq 1}$ a sequence of $\mathbb R$-trees. For any $i \geq 1$,  consider the rescaled tree $p_i^{\beta}\tau_i$, i.e. $(\tau_i,p_i^{\beta}d_i, \rho_i)$, and attach $p_i^{\beta}\tau_i$ to the point $x_i$ of the atom $p_i$ by identifying the root $\rho_i$ with the point $x_i$. More formally, define
\begin{equation}
\left(\tau', d', \rho'\right) \label{attach1}
\end{equation}
by taking the disjoint union $\tau':= [0,\ell] \sqcup \bigsqcup_{i \geq 1} \tau_i \setminus \{\rho_i\}$ and the metric $d'$ on $\tau'$ given by
\begin{equation}
d'(x,y):= \begin{cases} \left|x-y\right|   &\hspace{-0.2cm}\text{if } x,y \in [0,\ell], \\
p_j^\beta d_j(x,y) &\hspace{-0.2cm}\text{if } x,y \in \tau_j, j \geq 1, \\
\left|x-x_j\right| + p_j^\beta d_j(\rho_j, y) &\hspace{-0.2cm}\text{if } x \in [0,\ell], y \in \tau_j, j \geq 1, \\
p_{j_1}^\beta d_{j_1}(x,\rho_{j_1}) +\left|x_1-x_2\right|+ p_{j_2}^\beta d_{j_2}(\rho_{j_2}, y) &\hspace{-0.2cm}\text{if } x \in \tau_{j_1}, y \in \tau_{j_2}, j_1 \neq j_2, \end{cases} \label{attach2} \end{equation}  
where we define the root by $\rho':=0$. We only consider $(\tau', d', \rho')$ when compact. It is easy to see that the equivalence class of $(\tau', d', \rho')$ only depends on the equivalence classes of $(\tau_i, d_i, \rho_i)$, $i \geq 1$, hence we can define the subset $C_\beta\subset \Xi^*$ as the set of all elements $\vartheta=(\xi,\tau_i,i \geq 1) \in \Xi^*$ such that $(\tau',d',\rho')$ is compact (for any representatives), and the map $\phi_\beta \colon \Xi^* \rightarrow \mathbb T$,
\begin{equation*} \vartheta:=([0,\ell], (x_i)_{i  \geq 1}, (p_i)_{i  \geq 1},(\tau_i)_{i  \geq 1}) \mapsto {\phi}_\beta ([0,\ell], (x_i)_{i  \geq 1}, (p_i)_{i \geq 1}, (\tau_i)_{i \geq 1}) \end{equation*}
mapping $\vartheta=(\xi,\tau_i,i\ge 1)$ to the equivalence class in $\mathbb T$ associated with $(\tau', d', \rho')$  if $\vartheta \in  C_\beta$, and to the equivalence class of the one-point tree $(\{\rho\},0,\rho)$ otherwise. 

%As we have a natural injection between $\widetilde{\Xi}$ and $\Xi$, we will also write $\phi_\beta(\xi, (\tau_i, i \geq 1))$ for any $\xi=([0,K], \sum_{i \geq 1}P_i \delta_{x_i}, \lambda) \in \widetilde{\Xi}_{s}$ to mean $\phi_\beta(K, (x_i)_{i\geq 1}, (P_i)_{i \geq 1}, \lambda, (\tau_i, i \geq 1)).$ Note that
%$$ \phi_\beta\left(K, \left(x_i\right)_{i\geq 1}, \left(P_i\right)_{i \geq 1}, \lambda, \left(\tau_i, i \geq 1\right)\right) = \phi_\beta \left(K, \left(x_i\right)_{i\geq 1}, \left(P_i\right)_{i \geq 1}, 0, \left(\tau_i, i \geq 1\right)\right).$$

\begin{prop} \label{meas} The map $\phi_\beta: \Xi^* \rightarrow \mathbb T$ is Borel measurable. \end{prop}

The proof of Proposition \ref{meas} is based on an elementary lemma, whose proof we provide for completeness.

\begin{lemma}\label{help} Let $(A, \mathcal A)$ and $(B, \mathcal B)$ be two separable metrisable topological spaces equipped with their Borel $\sigma$-algebras, and let 
  $A_0 \in \mathcal A$ and $b_0\in B$. Then any function $f: A \rightarrow B$ with $f(a)=b_0$ for all $a \in A \setminus A_0$ and $f \restriction_{A_0}$ continuous is 
  Borel measurable.  
\end{lemma}
\begin{proof} First consider the case $A_0=A$ so that $f$ is continuous. Consider the collection of sets
\begin{equation*} \mathcal G:=\left\{ B_0 \in \mathcal B: f^{-1}(B_0) \in \mathcal A \right\}. \end{equation*}
By continuity, the pre-image $f^{-1}(B_0)$ of any open set $B_0 \in \mathcal B$ is open, and hence $\mathcal A$-measurable, i.e. $B_0\in \mathcal G$ for all open sets $B_0 \in \mathcal B$. Furthermore, by properties of the pre-image function $f^{-1}$, the collection of sets $\mathcal G$ is a $\sigma$-algebra. Since the Borel $\sigma$-algebra $\mathcal B$ on $B$ is generated by all open sets, we conclude that $\mathcal B=\mathcal G$, i.e. $f$ is Borel measurable.

  Now assume $A_0\neq A$. Then the subset topology on $A_0$ is separable and induced by any metric that generates $(A,\mathcal A)$, restricted to $A_0$. By the first case, $f_0=f\restriction_{A_0}$ is
  Borel measurable. We need to show that $f^{-1}(B_0) \in \mathcal A$ for any $B_0 \in \mathcal B$. We have that
\begin{equation*}
f^{-1}(B_0) = \begin{cases} \left(f^{-1}(b_0)\setminus f_0^{-1}(b_0) \right) \cup f_0^{-1}(B_0) & \text{ if } b_0 \in B_0, \\
f_0^{-1}(B_0)  & \text{ if } b_0 \notin B_0. \end{cases}
\end{equation*}
Since $f_0$ is Borel measurable, we obtain $f_0^{-1}(B_0) \in \mathcal A$. Furthermore, $f^{-1}(b_0) \setminus f_0^{-1}(b_0)=A\setminus A_0 \in \mathcal A$ since $A, A_0 \in \mathcal A$. Thus, $f^{-1}(B_0) \in \mathcal A$ for all $B_0 \in \mathcal B$, i.e. $f$ is Borel measurable.
\end{proof}

\begin{proof}[Proof of Proposition \ref{meas}] To apply Lemma \ref{help} to $A=\Xi^*$, $B=\mathbb T$, $f=\phi_\beta$, $A_0=C_\beta$ and $b_0$ the equivalence 
  class of the one-point 
  tree $(\{\rho\},0,\rho\})$, we check the assumptions of Lemma \ref{help}. The topological assumptions on the spaces hold by Propositions \ref{sep} and \ref{sepcompt}.
  We next turn to the continuity of ${\phi}_\beta \restriction_{C_\beta}$. 

Let $(\vartheta^{(n)})_{n \geq 1}$ be a sequence in $ C_\beta$ such that $\lim_{n \rightarrow \infty} d_{\beta}( \vartheta^{(n)}, \vartheta) = 0$ for some $\vartheta \in C_\beta$. Then, using the notation \begin{equation*} \vartheta^{(n)}=\left([0,\ell^{(n)}], \left(x_i^{(n)}\right)_{i \geq 1}, \left(p_i^{(n)}\right)_{i \geq 1}, \left(\tau_i^{(n)}\right)_{i \geq 1}\right), \quad n \geq 1, \end{equation*} and $\vartheta=([0,\ell],(x_i)_{i \geq 1}, (p_i)_{i \geq 1}, (\tau_i)_{i \geq 1})$, we obtain \begin{align*}
d_{\rm GH}\left( \phi_\beta\left(\vartheta^{(n)}\right) , \phi_\beta\left(\vartheta\right) \right)&\leq \left|\ell^{(n)}- \ell \right| + \sup_{i \geq 1} \left| x_i^{(n)}-  x_i \right| +\sup_{i \geq 1} d_{\rm GH}\left(\left(p_i^{(n)}\right)^\beta \tau_i^{(n)}, \left(p_i\right)^\beta \tau_i  \right) \\
&\leq 3 d_{\beta}\left(\vartheta^{(n)}, \vartheta\right),
\end{align*} 
i.e. $\lim_{n \rightarrow \infty} d_{\rm GH}( \phi_\beta(\vartheta^{(n)}) , \phi_\beta(\vartheta))=0$, and hence, ${\phi}_\beta\restriction_{C_\beta}$ is continuous. 
 
  It remains to show that $C_\beta \in \mathcal B(\Xi^*)$. For $\left([0,\ell], (x_i)_{i\geq 1}, (p_i)_{i \geq 1}, (\tau_i)_{i \geq 1} \right) \in \Xi^*$, the representatives of ${\phi}_\beta \left([0,\ell], (x_i)_{i\geq 1}, (p_i)_{i \geq 1},(\tau_i)_{i \geq 1}\right)$ are compact if and only if for all $N \geq 1$, we can find $I \geq 1$ such that for all $i \geq I$, ${\rm ht}(p_i^\beta \tau_i) < 1/N$, i.e. 
\begin{equation*}  C_\beta={\bigcap \limits_{N \geq 1} \bigcup \limits_{I \geq 1} \bigcap\limits_{i \geq I} } \{ \left([0,\ell], (x_i)_{i \geq 1}, (p_i)_{i \geq 1}, (\tau_i)_{i \geq 1}\right) \in \Xi^*: p_i^\beta {\rm ht}(\tau_i) < 1/N \}. 
\end{equation*}
Since the function ${\rm ht}\colon \mathbb T \to  [0, \infty)$ is Borel measurable by continuity, $C_\beta$ is $\cB(\Xi^*)$-measurable. 
\end{proof}
Related grafting operations have been studied in various tree formalisms. See e.g.\ \cite{16,DuWi1,ADH15}.%\pagebreak

To complete the setup for the contraction method stated as Lemma \ref{contrac}, consider the set of probability measures $\mathcal P(\mathbb T)$ on the space $\mathbb T$ of equivalence classes of rooted compact $\mathbb R$-trees, and for $p\ge 1$, the subset $\mathcal P_p \subset \mathcal P(\mathbb T)$ given by
\begin{equation} \mathcal P_{p}:=\left\{\eta\in\mathcal P(\mathbb T)\colon\mathbb E\left[{\rm ht}(\tau)^{p}\right] < \infty \text{ for } \tau \sim \eta \right\}. \end{equation}

We follow \cite{14} and equip $\mathcal P_p$ with the \textit{Wasserstein metric} of order $p \geq 1$, which is defined by
\begin{equation}
W_p\left(\eta, \eta'\right) :=\left(\inf \mathbb E\left[\left|d_{\rm GH}\left(\tau, \tau'\right)\right|^p\right]\right)^{1/p}, \qquad \eta, \eta' \in \mathcal P_p,
\end{equation}
where the infimum is taken over all joint distributions of $(\tau,\tau^\prime)$ on $\mathbb T^2$ with marginal distributions $\tau\sim\eta$ and $\tau^\prime\sim\eta'$. The space $(\mathcal P_p, W_{p})$ is complete since $d_{\rm GH}$ is a complete metric on $\mathbb T$, see e.g. \cite{W1, 21}. Convergence in 
$(\mathcal P_p,W_p)$ implies weak convergence on $(\bT,d_{\rm GH})$ and convergence of $p$th tree height moments.

\subsection{Fixpoints, construction of random $\mathbb R$-trees and Gromov-Hausdorff limits}\label{fixsec}

Let $\xi=([0,L],(X_i)_{i\ge 1},(P_i)_{i\ge 1})$ be any $\Xi$-valued random variable and $N=\inf\{i\ge 0\colon P_{i+1}=0\}$ the number of non-zero atom masses $P_i$ of 
$\xi$, with the convention that $\inf\varnothing=\infty$. In this section, we will study the recursive tree framework 
$((\xi_{\mathbf i},N_{\mathbf i},i\in\mathbb U);\phi_\beta)$, in which the $(\xi_{\mathbf i},N_{\mathbf i})$, $\mathbf{i}\in\mathbb U$, are i.i.d. copies of $(\xi,N)$. Since
$N$ is a function of $\xi$, we will slightly abuse notation and terminology in denoting by $\nu$ the distribution of $\xi$ rather than $(\xi,N)$ and in referring to 
$((\xi_{\mathbf i},i\in\mathbb U);\phi_\beta)$ or $(\xi_{\mathbf i},i\in\mathbb U)$ as the RTF. 

\begin{lemma}[Contraction] \label{contract}  Let $\beta>0$, $p\ge 1$ and $([0,L], (X_i)_{i \geq 1}, (P_i)_{i \geq 1})\in \Xi$ with some distribution $\nu$ such that $\bE[L^p]<\infty$ and $\mathbb E[\sum_{j \geq 1}P_j^{p\beta}] < 1$. Then the map $\Phi_\beta \colon \mathcal P_p \rightarrow \mathcal P_p$ associated with 
$\phi_\beta: \Xi^* \rightarrow \mathbb T$ is a strict contraction with respect to the Wasserstein metric of order $p$, i.e. 
\begin{equation}
\sup \limits_{\eta, \eta' \in \mathcal P_p, \eta \neq \eta'} \frac{W_p \left(\Phi_\beta(\eta), \Phi_\beta(\eta')\right)}{W_p(\eta, \eta')} < 1.
\end{equation}
\end{lemma}

\begin{proof} Let $\xi=([0,L], (X_i)_{i \geq 1}, (P_i)_{i \geq 1})\sim\nu$ and $(\tau_i)_{i \geq 1}$ an i.i.d. sequence with $\tau_1\sim\eta\in\mathcal P_p$ independent of $\xi$. First note that
\begin{align*}\mathbb E\left[\left({\rm ht}\left(\phi_\beta\left(\xi,\tau_i,i\ge 1\right)\right)\right)^p\right]
  &\le\mathbb E\left[\left(L+\sup_{i\ge 1}{\rm ht}\left(P_i^\beta\tau_i\right)\right)^p\right]
  \le 2^p\mathbb E[L^p]+2^p\mathbb E\left[\sup_{i\ge 1}P_i^{p\beta}\left({\rm ht}(\tau_i)\right)^p\right],
\end{align*} 
where the first term is finite by assumption, and the second term can be further estimated above using
\begin{align*}\mathbb E\left[\sup_{i\ge 1}P_i^{p\beta}\left({\rm ht}(\tau_i)\right)^p\right]
  &\le \mathbb E\left[\sum_{i\ge 1}P_i^{p\beta}\left({\rm ht}(\tau_i)\right)^p\right]
  =\left(\sum_{i\ge 1}\mathbb E\left[P_i^{p\beta}\right]\right)\mathbb E\left[\left({\rm ht}(\tau_1)\right)^p\right],
\end{align*}
which is also finite by assumption and since $\eta\in\mathcal P_p$. Now suppose $(\tau'_i)_{i \geq 1}$ is another i.i.d. sequence independent of $\xi$ with $\tau_1'\sim\eta'\in\mathcal P_p$. Then
\begin{align*}
W^p_p\left(\Phi_\beta\left(\eta\right), \Phi_\beta\left(\eta'\right)\right) &\leq \inf \mathbb E\left[\left(d_{\rm GH}\left(\phi_\beta\left(\xi, \tau_i, i \geq 1\right), \phi_\beta\left(\xi, \tau_i', i\geq 1 \right)\right)\right)^p\right]\\
&\leq \inf  \mathbb E\left[ \sup_{i \geq 1} \left( d_{\rm GH} \left(P_i^\beta \tau_i, P_i^\beta \tau_i'\right)\right)^p \right],
\end{align*}
where the infimum is taken over all couplings of $(\tau_i)_{i \geq 1}$ and $(\tau'_i)_{i \geq 1}$. An argument similar to the one above for heights now yields for the Gromov-Hausdorff distances 
\begin{align}
\inf  \mathbb E\left[ \sup_{i \geq 1} \left( d_{\rm GH} \left(P_i^\beta \tau_i, P_i^\beta \tau_i'\right)\right)^p \right]
%&\leq \inf \mathbb E\left[\sum_{i=1}^\infty d_{\rm GH}\left(P_i^\beta \tau_i, P_i^\beta \tau_i'\right)^p \right] \nonumber \\
&\leq \inf \mathbb E\left[\sum_{i\ge 1} \left(P_i\right)^{p\beta} d_{\rm GH}\left( \tau_i, \tau_i'\right)^p \right] \nonumber\\
%&=\inf \sum_{i=1}^\infty \mathbb E\left[ \left(P_i\right)^{p\beta}\right]\mathbb E\left[ d_{\rm GH}\left( \tau_i, \tau_i'\right)^p \right]\nonumber \\
&= \left(\sum_{i\ge 1} \mathbb E\left[ \left(P_i\right)^{p\beta}\right]\right) \left(\inf\mathbb E\left[ d_{\rm GH}\left( \tau_1, \tau_1'\right)^p \right]\right)\nonumber\\ 
&= \left(\sum_{i\ge 1}\mathbb E\left[ \left(P_i\right)^{p\beta}\right]\right)W^p_p\left(\eta, \eta'\right)
\end{align}
where we bounded above the infimum over all couplings of $(\tau_i)_{i \geq 1}$ and $(\tau'_i)_{i \geq 1}$ by the infinimum over those couplings for which $(\tau_i,\tau'_i)$ are i.i.d..

Dividing by $W^p_p(\eta, \eta')$, taking the $1/p$-th power and the supremum over all $\eta, \eta' \in \mathcal P_p$, $\eta \neq \eta'$, we obtain 
\begin{equation} \sup \limits_{\eta, \eta' \in \mathcal P_p, \eta \neq \eta'} \frac{W_p(\Phi_\beta(\eta), \Phi_\beta(\eta'))}{W_p(\eta, \eta')} \leq  \left(\sum_{i\ge 1} \mathbb E\left[ \left(P_i \right)^{p\beta} \right]\right)^{1/p}<1.\end{equation}
This completes the proof.
\end{proof}

The following corollary proves Theorem \ref{uniquefix}.

\begin{corollary}[Fixpoint] \label{fixpoint} In the setting of Lemma \ref{contract}, the map $\Phi_\beta: \mathcal P_p \rightarrow \mathcal P_p$ has a unique fixpoint ${\eta}^* \in \mathcal P_p$, i.e. $\Phi_\beta({\eta}^*)={\eta}^* $, and $\Phi_{\beta}^n(\eta) \rightarrow {\eta}^* $ in $(\mathcal P_p, W_p)$, as $n \rightarrow \infty$, for all $\eta \in \mathcal P_{p}$.
\end{corollary}
\begin{proof} Recall \cite{W1} that $W_p$ is a complete metric on $\mathcal P_{p}$ for any $1 \leq p < \infty$. Furthermore, $\Phi_\beta$ is a strict contraction by Lemma \ref{contract}. Hence, we conclude by Lemma \ref{contrac}, or directly by Banach's fixpoint theorem, that $\Phi_\beta \colon \mathcal P_{p} \rightarrow \mathcal P_{p}$ has a unique fixpoint ${\eta}^* \in \mathcal P_{p}$ such that $\Phi_{\beta}^{n}(\eta) \rightarrow {\eta}^*$ as $n \rightarrow \infty$ for all $\eta \in \mathcal P_{p}$. 
\end{proof}

\begin{corollary}[Moments]\label{heightmom} In the setting of Lemmas \ref{contract} and \ref{fixpoint}, a random $\mathbb R$-tree $\check{\mathcal T} \sim \eta^*$ has finite moments $\mathbb E[{\rm ht}(\check{\mathcal T})^p] < \infty$ for all $p<p^*=\sup\{p \geq 1\colon\bE[\sum_{i \geq 1}(P_i)^{p\beta}]<1\mbox{ and }\mathbb E[L^p]<\infty\}$.
\end{corollary}
\begin{proof} Note that, as a direct consequence of Lemma \ref{contract}, $\eta^* \in \bigcap_{p\ge 1\colon\bE[\sum_{i \geq 1}(P_i)^{p\beta}]<\infty,\bE[L^p]<\infty} \mathcal P_p$.
\end{proof}

In Lemma \ref{contract}, we assume that there exists $p \geq 1$ such that $\mathbb E[\sum_{i \geq 1} P_i^{p\beta}] < 1$. In the case of a $\Xi_{\rm s}$-valued random 
string of beads or a $\Xi_{\rm g}$-valued random generalised string $([0,L],(X_i)_{i\ge 1},(P_i)_{i\ge 1},\Lambda)$, this condition holds if (and only if, when
$\Lambda=0$) we have $p>1/\beta$:
\begin{equation}
\sum_{i\ge 1} \mathbb E\left[ \left(P_i\right)^{p \beta}\right] < \sum_{i\ge 1}\bE[P_i]=\bE\left[\sum_{i\ge 1} P_i\right]\le 1, \label{pbetha}
\end{equation}
so the condition in Lemma \ref{contract} boils down to $\mathbb E[L^p]<\infty$ for some $p\ge 1$ with $p>1/\beta$. In particular, in this case, ${\rm ht}(\check{\mathcal T})$ has moments of all orders if and only if $L$ has moments of all orders.

%Corollary \ref{fixpoint} yields a distribution ${\eta}^*$ on $\mathbb T$, and a recursive construction of an $\mathbb R$-tree $\check{\mathcal T} \sim {\eta}^*$ based on an element of ${\Xi}$ with distribution ${\nu}$, cf. \cite{14}. 

Let us turn to the recursive construction of a random $\mathbb{R}$-tree with the fixpoint distribution $\eta^*$ from a recursive tree framework $((\xi_{\mathbf i},\mathbf{i}\in\mathbb U),\phi_\beta)$ where the $\xi_{\mathbf i}$ are independent copies of $\xi=([0,L],(X_i)_{i\ge 1},(P_i)_{i\ge 1})$.  
%Specifically, we write 
%\begin{equation}\left([0,L_{\mathbf i}], (X_{\mathbf{i}j})_{j \geq 1},  (P_{\mathbf{i}j})_{j \geq 1}\right) 
%\end{equation} 
%i.i.d. copies of some random $\xi \in \Xi$. We denote by $E_{\mathbf i}$ isometric copies of $[0,L_{\mathbf{i}}]$, 
%These induce isometric copies of i.i.d. generalised strings which we denote by \begin{equation}
%\left(E_{\mathbf i}, \sum_{j \geq 1} P_{{\mathbf i }j} \delta_{x_{{\mathbf i}j}}\right), \quad \mathbf{i} \in \mathbb U, \end{equation} 
%where $E_{\mathbf i}, \mathbf{i} \in \mathbb U$, are disjoint intervals, and we abuse notation to write $X_{\mathbf{i}j}\in E_{\mathbf i}$ for the
%atoms mapped to $E_{\mathbf i}$ under the isometry, and we write 
%\begin{equation} \xi_{\mathbf i}=\left(E_{\mathbf i}, (X_{\mathbf{i}j})_{j\ge 1},(P_{{\mathbf i }j})_{j\ge 1}\right), \quad \mathbf{i} \in \mathbb U. 
%\end{equation} 
%We will further abuse notation by writing $\xi_{\mathbf i} \subset \mathcal T$ to mean that $E_{\mathbf i} \subset \mathcal T$ for some $\mathbf{R}$-tree $\mathcal T$.
In Corollary \ref{recconcrt}, we want to construct a sequence of trees $(\check{\mathcal T}_n, n \geq 0)$ by successively ``replacing the atoms'' on  
$\check{\mathcal T}_n$ by rescaled $\xi_{\mathbf i}$. The a.s.\ Gromov-Hausdorff limit $\check{\mathcal T}$ of this sequence is then identified as  having the 
fixpoint distribution $\eta^*$ and is fully determined by  $(\xi_{\mathbf i}, \mathbf{i} \in \mathbb U)$, a property called \em endogeny \em in \cite{14}.
%Note that, by construction, $\check{\mathcal T}=\overline{\bigcup_{n \geq 0} \check{\mathcal T}_n}=\overline{\bigcup_{\mathbf{i} \in \mathbb U} E_{\mathbf i}}$. 

The following result restates formally the above construction and establishes Theorem \ref{constr3} of the introduction.

\begin{prop}[Recursive construction] \label{recconcrt} Let $p\ge 1$ and $([0,L], (X_i)_{i \geq 1}, (P_i)_{i \geq 1})\in \Xi$ with some distribution $\nu$ such that 
  $\bE[L^p]<\infty$ and $\mathbb E[\sum_{j \geq 1}P_j^{p\beta}] < 1$. We construct an increasing sequence of random 
  $\mathbb R$-trees $(\check{\mathcal T}_n)_{n \geq 0}$,  as follows. 
  Let $\xi_{\mathbf i}=([0,L_{\mathbf{i}}], (X_{\mathbf{i}j})_{j \geq 1},  (P_{\mathbf{i}j})_{j \geq 1})$, $\mathbf{i} \in \mathbb U$, be independent copies of $\xi$.  
  For $n \geq 0$, set $\check{\mathcal T}_n:=\tau_\varnothing^{(n)}$, where
  $(\tau^{(n)}_{\mathbf i}, \mathbf{i} \in \bigcup_{k=0}^{n} \mathbb N^{k} )$ is a sequence of trees defined recursively by 
  \begin{equation} \tau^{(n)}_{{\mathbf{i}}}:=
    \begin{cases} [0,L_{\mathbf i}], & \mathbf{i} \in \mathbb N^{n}, \\
				  \phi_\beta\left(\xi_{\mathbf{i}}, \left(\tau^{(n)}_{\mathbf{i} j} , j \geq 1\right)\right) & \mathbf{i} \in \mathbb N^{m}, n-1\geq m \geq 0. 
    \end{cases}\label{taucon}  
  \end{equation} 
  Then we have the convergence
  \begin{equation} \lim \limits_{n \rightarrow \infty} d_{\rm GH}\left(\check{\mathcal T}_n, \check{\mathcal T} \right)=0 \quad \text{ a.s. } 
  \end{equation}
  where the limiting random ${\mathbb R}$-tree $\check{\mathcal T}$ has the fixpoint distribution ${\eta}^*$ associated with the RTF
  $(\xi_{\mathbf i}, \mathbf{i} \in \mathbb U)$.
\end{prop}

\begin{proof} Denote by $\eta_0$ the distribution of (the equivalence class) of a one-branch $\mathbb R$-tree of length $L$. By construction, $\tau_{\mathbf{i}}^{(n)}\sim\eta_0$ for all $\mathbf{i}\in\mathbb N^n$, $n\ge 0$. Recursively, we see that $\tau_{\mathbf{i}}^{(n)}\sim\eta_{n-m}$ for all $\mathbf{i}\in\mathbb N^m$, $n\ge m\ge 0$, where $\eta_m=\Phi_\beta^m(\eta_0)$. In particular $\check{\mathcal T}_n=\tau_\varnothing^{(n)}\sim\eta_n\rightarrow\eta^*$
since $\eta^*$ is the unique attractive fixpoint of Corollary \ref{fixpoint}. Since Wasserstein convergence implies the convergence of moments of tree heights, we
have $M:=\sup_{n\ge 0}\mathbb E[{\rm ht}(\check{\mathcal T}_n)^p]<\infty$.

On the other hand, the sequence $(\check{\mathcal T}_n, n \geq 0)$ has been coupled via the underlying RTF $(\xi_{\mathbf i}, \mathbf{i} \in \mathbb U)$. More precisely, $\check{\mathcal T}_n \subset \check{\mathcal T}_{n+1}$ for all $n \geq 0$,
if we represent the strings of beads $\xi_{\mathbf i}$ as disjoint (half-open) intervals $E_{\mathbf i}$, $\mathbf i\in\mathbb{U}$, and define a metric on their disjoint 
union $\check{\mathcal T}_n=[0,L_\varnothing]\sqcup\bigsqcup_{\mathbf{i}\in\bigcup_{m=1}^n\mathbb N^m}E_{\mathbf i}$ that captures the repeated scaling of subtrees via $\phi_\beta$ in (\ref{taucon}). Specifically,
the scaling factor for $E_{\mathbf{i}}$, when directly attached to $X_{{\mathbf{i}}}\in\check{\mathcal T}_n$ is
\begin{equation} \check{P}_{\mathbf{i}} = P_{i_1} P_{i_1i_2} \cdots P_{ i_1i_2\ldots i_{n+1}},\quad \mathbf{i}=i_1i_2\ldots i_{n+1} \in \mathbb N^{n+1}, \quad n \geq 0.\label{pcheck} \end{equation}
We note that e.g. $E_{\mathbf i}\subset\check{\mathcal T}_{n+1}\setminus\check{\mathcal T}_n$, $\mathbf i\in\mathbb{N}^{n+1}$, appears unscaled in $\tau^{(n+1)}_{\mathbf{i}}$,
scaled by $P_{\mathbf{i}}^\beta$ in $\tau^{(n+1)}_{i_1\ldots i_n}\supset\tau^{(n)}_{i_1\ldots i_n}$ and, inductively, scaled by $\check{P}_{\mathbf{i}}^\beta$ in 
$\check{\mathcal T}_{n+1}=\tau^{(n+1)}_\varnothing\supset\tau^{(n)}_\varnothing=\check{\mathcal T}_n$, always attached at $X_{\mathbf i}\in E_{i_1\ldots i_n}\subset\check{\mathcal T}_n\subset\check{\mathcal T}_{n+1}$.

%Therefore, the tree $\check{\mathcal T}$ is well-defined and has the fixpoint distribution ${\eta}^*$ of Corollary \ref{fixpoint}. By the same argument, the trees
%$\mathcal T_i=\overline{\bigcup_{n\ge 0}\tau_i^{(n+1)}}$ have the fixpoint distribution, they are independent and satisfy
%\begin{equation} \check{\mathcal T} = \phi_\beta\left(\xi_{\varnothing}, (\mathcal T_{i}, i \geq 1) \right). 
%\end{equation} 
%Similarly, $\mathcal T_{\mathbf{i}} = \phi_\beta(\xi_{\mathbf i}, (\mathcal T_{\mathbf{i}j}, j \geq 1) ),  \mathbf{i} \in \mathbb U.$ For any $p \geq 1$ with $\sum_{j\ge 1}\mathbb{E}(P_j^{\beta p})< 1$ and $\mathbb E(L^p)<\infty$, we have
To deduce the almost sure convergence of $(\check{\mathcal T}_n)_{n\ge 1}$, we could now apply general results \cite[Theorems 3.3 and 3.9]{DuWi2} for sequences of random $\mathbb{R}$-trees that are increasing with respect to a partial order $\preceq$ on $(\mathbb{T},d_{\rm GH})$ that captures the inclusion property of suitable representatives. However, a direct 
argument is straightforward and follows on nicely from the proof of our Lemma \ref{contract}. For all $n>m\ge 0$
\begin{align*}  \mathbb E\left[ d_{\rm GH} \left( \check{\mathcal T}_m, \check{\mathcal T}_n\right)^{p}\right] \leq \mathbb E \left[\left(\sup_{\mathbf{i} \in \mathbb N^{m}, j \geq 1}   \check{P}_{\mathbf{i}j}^\beta {\rm ht}(\tau_{\mathbf{i}j}^{(n)}) \right)^{p\,}\right] 
&\leq \sum_{\mathbf{i} \in \mathbb N^{m}}\mathbb E \left[\sum_{j\geq1} \check{P}_{\mathbf{i}j}^{p\beta} \right] \mathbb E \left[{\rm ht}( \check{\mathcal T}_{n-m-1})^{p}\right]. \end{align*}
Note that, for each $\mathbf{i} \in \mathbb U$, the sequence $(\check{P}_{\mathbf{i}j}, j \geq 1)$ can be represented as 
\begin{equation} \left(\check{P}_{\mathbf{i}j}, j \geq 1\right) = \left(\check{P}_{\mathbf{i}}  \left({P}_{\mathbf{i}j}, j \geq 1\right)\right) \label{PDsplit} \end{equation}
where $({P}_{\mathbf{i}j}, j \geq 1)$ is independent of $\check{P}_{\mathbf{i}}$, $\mathbf{i} \in \mathbb N^{m}$, and has the same distribution as the ranked atom masses $(P_j, j \geq 1)$ of the initial $\xi_{\varnothing}$. Therefore, 
\begin{align*} \mathbb E\left[ d_{\rm GH} \left( \check{\mathcal T}_m,  \check{\mathcal T}_n \right)^p\right]  &\leq\mathbb E \left[{\rm ht}( \check{\mathcal T}_{n-m-1})^{p}\right] \mathbb E \left[  \sum_{j \geq 1} P_j^{p\beta} \right] \mathbb E \left[ \sum_{\mathbf{i} \in \mathbb N^{m}} \check{P}_{\mathbf{i}}^{p\beta}\right].
\end{align*}
Applying \eqref{PDsplit} recursively, we obtain
\begin{align} \mathbb E\left[ d_{\rm GH} \left( \check{\mathcal T}_m, \check{\mathcal T}_n \right)^p\right]  &\leq\mathbb E \left[ {\rm ht}\left( \check{\mathcal T}_{n-m-1}\right)^{p}\right] \left(\mathbb E \left[  \sum_{j \geq 1} P_j^{p\beta} \right]\right)^{m+1}\leq M\left(\mathbb E \left[  \sum_{j \geq 1} P_j^{p\beta} \right]\right)^{m+1}. \label{bound}
\end{align}
We have $\mathbb E \left[  \sum_{j \geq 1} P_j^{p\beta }\right]<1$ and $d_{\rm GH}(\check{\mathcal T}_m,\check{\mathcal T}_n)$ increasing in $n\ge m+1$ for fixed $m\ge 0$, so
\begin{align*}
\mathbb P\left(d_{\rm GH}\left(\check{\mathcal T}_m,\check{\mathcal T}_n\right)>\epsilon\text{ for any }n\ge m+1\right)&=
\sup_{n\ge m+1}\mathbb P \left( d_{\rm GH}\left(\check{\mathcal T}_m,  \check{\mathcal T}_n\right) > \epsilon\right)\\ &\leq \sup_{n\ge m+1}\epsilon^{-p} \mathbb E\left[ d_{\rm GH} \left( \check{\mathcal T}_m, \check{\mathcal T}_n \right)^p \right] \leq M \epsilon^{-p} \left(\mathbb E \left[  \sum_{j \geq 1} P_j^{p\beta} \right]\right)^{m+1},  
\end{align*}
and hence, by the first Borel-Cantelli Lemma, we conclude
\begin{equation}
\mathbb P \left( d_{{\rm GH}}\left(\check{\mathcal T}_m, \check{\mathcal T}_n \right) > \epsilon \text{ for any }n\ge m+1\text{, for infinitely many } m \right) =0.
\end{equation}
Consequently, $(\check{\mathcal T}_n)_{n\ge 0}$ is $d_{\rm GH}$-Cauchy a.s., so $\lim_{n \rightarrow \infty} \check{\mathcal T}_n =  \check{\mathcal T}$ exists a.s. in the Gromov-Hausdorff topology.
\end{proof}

\subsection{Construction of mass measures and the Gromov-Hausdorff-Prokhorov limits} \label{resultsstrings}

The methods we use are robust to changes of formalism. In the previous two sections, we worked in $(\mathbb T,d_{\rm GH})$, while the same arguments actually 
work in $(\mathbb T^{\rm emb},\delta_{\rm H})$, provided that we define appropriate $\phi_\beta^{\rm emb}$ making explicit use of the tree structure of $\mathbb U$, or, 
with some restrictions, even on suitable spaces of real-valued excursions, see e.g.\ \cite{LG1,6,Duqcod}, provided that we represent strings of beads and generalised
strings accordingly. Let us not go into details here, but let us rephrase Proposition \ref{recconcrt} as a construction in $(\mathbb T^{\rm emb},\delta_{\rm H})$. This
will be useful when adding mass measures to our construction.

We need some notation that allows us to place a string $\xi_{\mathbf{i}}$ parallel to the $\mathbf{i}$th coordinate direction. Denote by $e_{\mathbf{i}}$ the unit vector in $l^1(\mathbb U)$ associated with coordinate $\mathbf{i}\in\mathbb{U}$, and by $\theta_j\colon l^1(\mathbb{U})\rightarrow l^1(\mathbb{U})$, for $j\in\mathbb N$, the coordinate shift operator
$$\theta_j\left((s_{\mathbf{i}})_{\mathbf{i}\in\mathbb{U}}\right)=\theta_j\left(\sum_{\mathbf{i}\in\mathbb{U}}s_{\mathbf{i}}e_{\mathbf{i}}\right)=\sum_{\mathbf{i}\in\mathbb{U}}s_{\mathbf{i}}\theta_j(e_{\mathbf{i}})=\sum_{\mathbf{i}\in\mathbb{U}}s_{\mathbf{i}}e_{j\mathbf{i}},$$
which we also apply element by element to subsets of $l^1(\mathbb U)$, specifically to $\mathbb{R}$-trees in $\mathbb{T}^{\rm emb}$. We denote by $\pi\colon\mathbb{T}^{\rm emb}\rightarrow\mathbb T$ the natural projection of an
$\mathbb R$-tree in $\mathbb{T}^{\rm emb}$ onto its equivalence class in $\mathbb{T}$.

\begin{corollary}[Recursive construction] \label{recconcrt2} Let $(\xi_{\mathbf i},\mathbf{i}\in\mathbb{U})$ be as in Proposition \ref{recconcrt}. We construct an
  increasing sequence of $\mathbb{T}^{\rm emb}$-valued random $\mathbb R$-trees $(\check{\mathcal T}_n^{\rm emb})_{n \geq 0}$, as follows. For $n \geq 0$, set 
  $\check{\mathcal T}_n^{\rm emb}:=\tau_\varnothing^{(n)}$, where $(\tau^{(n)}_{\mathbf i}, \mathbf{i} \in \bigcup_{k=0}^{n} \mathbb N^{k} )$ is a sequence of trees 
  defined recursively by 
  \begin{equation} \tau^{(n)}_{{\mathbf{i}}}:=
    \begin{cases} [0,L_{\mathbf i}]e_{\varnothing}, & \mathbf{i} \in \mathbb N^{n}, \\
				  [0,L_{\mathbf i}]e_{\varnothing}\cup\bigcup_{j\ge 1}\left(X_{\mathbf{i}j}e_\varnothing+\theta_j\left(P_{\mathbf{i}j}^\beta\tau^{(n)}_{\mathbf{i} j}\right)\right) & \mathbf{i} \in \mathbb N^{m}, n-1\geq m \geq 0. 
    \end{cases}\label{taucon2}  
  \end{equation} 
  Then we have the convergence
  \begin{equation*} \lim \limits_{n \rightarrow \infty} \delta_{\rm H}\left(\check{\mathcal T}_n^{\rm emb}, \check{\mathcal T}^{\rm emb} \right)=0 \quad \text{ a.s. } 
  \end{equation*}
  where the limiting random ${\mathbb R}$-tree $\check{\mathcal T}^{\rm emb}$ is such that $\pi(\check{\mathcal T}^{\rm emb})$ has the fixpoint distribution $\eta^*$. 
\end{corollary}
\begin{proof} The main point to check is that the second case of (\ref{taucon2}) is consistent with the application of $\phi_\beta$ in (\ref{taucon}). We leave the
  details to the reader.
\end{proof}

Let us now consider a $\Xi_{\rm g}$-valued random generalised string $\xi^\bullet=([0,L],(X_i)_{i\ge 1},(P_i)_{i\ge 1},\Lambda)$. We associate with $\xi^\bullet$ the
$\Xi$-valued random generalised string $\xi=([0,L],(X_i)_{i\ge 1},(P_i)_{i\ge 1})$ without the measure $\Lambda$ on $[0,L]$. Recall that the main restriction compared to 
Corollary \ref{recconcrt2} is that $\sum_{i \geq 1} P_i+ \Lambda([0,L])=1$ a.s. so that $\sum_{i\ge 1}P_i\delta_{X_i}+\Lambda$ is a probability measure on $[0,L]$. This 
is precisely what we need to show that the rescaled $n$th generation measures on $\check{\mathcal T}_n^{\rm emb}$ plus the untouched $\Lambda$-components of previous 
generations converge to a limiting probability measure on $\check{\mathcal{T}}^{\rm emb}$. Recall that (\ref{taucon2}) eventually places $[0,L_{\mathbf{i}}]$ parallel to 
$e_{\mathbf{i}}$. We will place the mass measure $\sum_{j\ge 1}P_{\mathbf{i}j}\delta_{X_{\mathbf{i}j}}+\Lambda_{\mathbf{i}}$ accordingly. Specifically, recursive scaling
by $P_{\mathbf{i}j}^\beta$ yields a total length scaling by $\check{P}_{\mathbf{i}}^\beta$ as defined in (\ref{pcheck}), which corresponds to recursive measure scaling by 
a total of $\check{P}_{\mathbf{i}}$, which preserves mass 1. Recursive shifting by $X_{\mathbf{i}j}$ places the beginning of the scaled $\xi_{\mathbf{i}}$ at 
\begin{equation}\label{Xcheck}\check{X}_{\mathbf{i}}=X_{i_1}e_\varnothing+P_{i_1}^\beta X_{i_1i_2}e_{i_1}+\cdots+P_{i_1}^\beta\cdots P_{i_1\ldots i_{m-2}}^\beta X_{\mathbf{i}}e_{i_1\ldots i_{m-1}}.
\end{equation}
Informally, we define the measures $\check{\mu}^{\rm emb}_n$ on $\check{\mathcal{T}}_n^{\rm emb}$, $n\ge 0$, by taking as $\check{\mu}_0^{\rm emb}$ the mass measure of $\xi_{\varnothing}^\bullet$ in direction $e_\varnothing$, and by building $\check{\mu}_{n+1}^{\rm emb}$ from $\check{\mu}_n^{\rm emb}$ by replacing each atom of size $\check{P}_{{\mathbf i}j}$ of $\check{\mu}_n^{\rm emb}$ by the rescaled mass measure of string $\xi_{\mathbf{i}j}^\bullet$ in direction $e_{\mathbf{i}j}$ starting from
$\check{X}_{{\mathbf i}j}$. In particular, the rescaled parts $\Lambda_{\mathbf{i}}$ of the mass measures are not replaced and remain parts of $\check{\mu}_n$ for all $n$.
The following result gives a formal definition of the mass measure $\check{\mu}^{\rm emb}$ on $\check{\mathcal{T}}^{\rm emb}$ as a limit of rescaled mass measures of the
$n$th generation strings with all $\Lambda$-measures from previous generations:

%\begin{corollary}\label{heightmom2} Let $\check{\mathcal T} \sim \eta^*$ be the tree with the fixpoint distribution $\eta^*$ for some $(K, (x_i)_{i \geq 1}, (P_i)_{i \geq 1}, \lambda)$ with distribution $\nu$ supported on $\Xi_g$. Then $\mathbb E[{\rm ht}(\check{\mathcal T})^p] < \infty$ for all $p >1/\beta$.
%\end{corollary}

%We further obtain Gromov-Hausdorff-Prokhorov convergence of the weighted $\mathbb R$-trees $(\check{\mathcal T}_n, \check{\mu}_n)$ in Corollary \ref{recconcrt}.

\begin{prop}[Mass measure on $ \check{\mathcal T}^{\rm emb}$] \label{mmreccon} Let $p\ge 1$, $\beta>0$, let $\xi^\bullet=([0,L],(X_i)_{i\ge 1},(P_i)_{i\ge 1},\Lambda)$ 
  be a $\Xi_{\rm g}$-valued random generalised string with $\mathbb{E}[L^p]<\infty$ and $\mathbb{E}[\sum_{j\ge 1}P_j^{p\beta}]<1$. Consider associated RTFs  
  $(\xi^\bullet_{\mathbf{i}},\mathbf{i}\in\mathbb{U})$, and $(\xi_{\mathbf{i}},\mathbf{i}\in\mathbb{U})$ without the measures $\Lambda_{\mathbf{i}}$. Let 
  $(\check{\mathcal{T}}^{\rm emb}_n,n\ge 0)$ and $\check{\mathcal{T}}^{\rm emb}$ be as in Corollary \ref{recconcrt2}. Set
  \begin{equation}\label{massmeas}
      \check{\mu}_{n}^{\rm emb}= \sum_{\mathbf{i}\in\mathbb N^{n},j\ge 1} \check{P}_{\mathbf{i}j}\delta_{\check{X}_{{\mathbf i}j}}
                                             + \sum_{m=0}^n \sum_{\mathbf{i}\in\mathbb N^m}\check{P}_{\mathbf{i}}\check{\Lambda}_{{\mathbf{i}}}\qquad\mbox{for all $n\ge 0$,}
  \end{equation}
  where $\check{\Lambda}_{\mathbf{i}}(\check{X}_{\mathbf{i}}+\check{P}_{\mathbf{i}}^\beta[a,b]e_{\mathbf{i}})=\Lambda_{\mathbf{i}}([a,b])$, $0\le a<b<\infty$, for all 
  $\mathbf{i}=i_1\ldots i_m\in\mathbb{U}$.        
  Then there is a random probability measure $\check{\mu}^{\rm emb}$ on $\check{\mathcal T}^{\rm emb}$ so that  
  \begin{equation}
    \lim_{n \rightarrow \infty} \delta_{\rm HP}\left(\left(\check{\mathcal T}_n^{\rm emb}, \check{\mu}_n^{\rm emb} \right), 
                                                     \left( \check{\mathcal T}^{\rm emb}, \check{\mu}^{\rm emb}\right) \right)= 0 \quad\text{a.s.}.
  \end{equation}
\end{prop}
\begin{proof} By Corollary \ref{recconcrt2}, $\lim_{n\rightarrow\infty}\delta_{\rm H}(\check{\mathcal T}_n^{\rm emb},\check{\mathcal T}^{\rm emb})=0$. More specifically,  
  \begin{equation}\label{expl} 
    \check{\mathcal T}^{\rm emb}_n = \bigcup_{m=0}^n\bigcup_{\mathbf{i}\in\mathbb{U}}\left(\check{X}_{\mathbf i}+\check{P}_{\mathbf{i}}^\beta[0,L_{\mathbf{i}}]e_{\mathbf{i}}\right)
    \quad\mbox{and hence}\quad\check{\mathcal T}^{\rm emb} = \overline{\bigcup_{n \geq 0} \check{\mathcal T}^{\rm emb}_n} 
      = \overline{\bigcup_{m\ge 0}\bigcup_{\mathbf{i}\in\mathbb{U}}\left(\check{X}_{\mathbf i}+\check{P}_{\mathbf{i}}^\beta[0,L_{\mathbf{i}}]e_{\mathbf{i}}\right)} 
  \end{equation}
  is the closure in $l^1(\mathbb{U})$ of an increasing union, a.s.\ compact as an a.s.\ Hausdorff limit of random compact subsets of $l^1(\mathbb{U})$. It remains to show that $\check{\mu}_n^{\rm emb}$ 
  converge a.s., as random probability measures on $l^1({\mathbb U})$. This is easier than Aldous's 
  \cite[Proof of Theorem 3]{8} approximation of the measure representation of the Brownian CRT, but the formalism and the key steps are the same. Firstly, mass is preserved from $\check{\mu}_n^{\rm emb}$ to $\check{\mu}_{n+1}^{\rm emb}$, since $\xi^\bullet$ is $\Xi_{\rm g}$-valued, and for all 
  $\mathbf{i}^{(1)},\ldots,\mathbf{i}^{(r)}\in\bigcup_{0\le m\le k}\mathbb{N}^m$, the projections $\pi_{\mathbf{i}^{(1)},\ldots,\mathbf{i}^{(r)}}\check{\mu}_n^{\rm emb}$
  onto the respective marginals of the random measures $\check{\mu}_n^{\rm emb}$, do not depend on $n\ge k$, hence converge a.s., to a consistent system of finite-dimensional
  marginals of a random probability measure $\check{\mu}^{\rm emb}$ on $[0,\infty)^{\mathbb{U}}$. And secondly, the sequence $(\check{\mu}_n^{\rm emb}, n \geq 0)$ is 
  (a.s.) tight as a family of (random) probability measures on $l^1(\mathbb{U})$, since $\check{\mathcal T}^{\rm emb}$ is compact a.s.. Hence, 
  $\check{\mu}^{\rm emb}:=\lim_{n \rightarrow \infty} \check{\mu}_n^{\rm emb}$ is a probability measure on $l^1(\mathbb U)$. 
\end{proof}

Denote by $\pi_{\rm w}\colon\mathbb{T}^{\rm emb}_{\rm w}\rightarrow\mathbb T_{\rm w}$ the natural projection of a weighted 
$\mathbb R$-tree in $\mathbb{T}^{\rm emb}_{\rm w}$ onto its equivalence class in $\mathbb{T}_{\rm w}$. We will abuse notation and write $(\mathcal{T},\mu)\in\mathbb{T}$, 
even though the elements of $\mathbb{T}$ are not weighted $\mathbb{R}$-trees but equivalence classes of weighted $\mathbb{R}$-trees. The point is that the operations 
described informally in Theorem \ref{constr2} of the introduction give rise to operations on $\mathbb{T}_{\rm w}$ just as $\phi_\beta$ was shown to be well-defined as an 
operation on $\mathbb{T}$. We leave the details to the reader. Since the projection $\pi_{\rm w}$ is $1$-Lipschitz, as noted in Proposition 
\ref{embprop}, we obtain the following corollary, which gives formal meaning to and establishes Theorem \ref{constr2}.

\begin{corollary}\label{eqmeas} In the setting of Proposition \ref{mmreccon}, consider the equivalence classes of weighted $\mathbb{R}$-trees 
  $(\check{\mathcal T},\check{\mu})=\pi_{\rm w}(\check{\mathcal T}^{\rm emb},\check{\mu}^{\rm emb})$ and
  $(\check{\mathcal T}_n,\check{\mu}_n)=\pi_{\rm w}(\check{\mathcal T}_n^{\rm emb},\check{\mu}_n^{\rm emb})$, $n\ge 0$, in $\mathbb{T}_{\rm w}$. Then
  \begin{equation}
    \lim_{n \rightarrow \infty} \delta_{\rm GHP}\left(\left(\check{\mathcal T}_n, \check{\mu}_n\right), 
                                                     \left( \check{\mathcal T}, \check{\mu}\right) \right)= 0 \quad\text{a.s.}.
  \end{equation}
\end{corollary}

When there is no mass left behind on the branches of the trees $\check{\mathcal T}_n$ in the update step to $(\check{\mathcal T}_{n+1},\check{\mu}_{n+1})$, that is, when $\Lambda=0$ a.s., and when the string does not
extend further than the atoms on the string, that is, when $L=\sup\{X_i\colon P_i>0,i\ge 1\}$ a.s., we may obtain a CRT $(\check{\mathcal T}, \check{\mu})$. The following corollary confirms this, and in particular, establishes Theorem \ref{introthm} from the introduction.

\begin{corollary}[Construction of a CRT] In the setting of Proposition \ref{mmreccon}, if $L=\sup\{X_i\colon P_i>0,i\ge 1\}$ and $\sum_{i \geq 1} P_i=1$ a.s., the random weighted 
  $\mathbb R$-tree $(\check{\mathcal T}^{\rm emb}, \check{\mu}^{\rm emb})$ is a CRT. 
\end{corollary}
\begin{proof}
We need to check that $\check{\mu}^{\rm emb}$ a.s. satisfies properties (i), (ii) and (iii) of a continuum tree given in Section \ref{Sec31}. Note that 
$\max_{x \in \check{\mathcal T}_n^{\rm emb}} \check{\mu}^{\rm emb}_n(x) = \max_{\mathbf{i}=i_1\cdots i_{n+1} \in \mathbb N^{n+1}} \check{P}_{\mathbf{i}}$, and that for $\epsilon > 0$, 
\begin{align*} \mathbb P\left( \max_{\mathbf{i}\in \mathbb N^{n}} \check{P}_{\mathbf i} >\epsilon \right) &\leq \epsilon^{-p\beta} \mathbb E \left[ \left(\max_{\mathbf{i}\in \mathbb N^{n}} \check{P}_{\mathbf i}\right)^{p\beta} \right]  \leq  \epsilon^{-p\beta} \mathbb E \left[ \sum_{\mathbf{i}\in \mathbb N^{n}} \check{P}_{\mathbf i}^{p\beta} \right]  \leq \epsilon^{-p\beta}  \left(\mathbb E \left[ \sum_{j \geq 1} {P}_{j}^{p\beta} \right] \right)^{n},
\end{align*}
where the second last inequality can be derived as in \eqref{PDsplit}. Since $\mathbb E [ \sum_{j \geq 1} {P}_{j}^{p\beta} ] <1$ for $p > 1/\beta$, these probabilities are summable over $n\ge 1$, and so $\lim_{n \rightarrow \infty} \max_{\mathbf{i} \in \mathbb N^n}  \check{P}_{\mathbf{i} } =0$ a.s.. By construction,
$$\check{\mu}_{m+1}^{\rm emb}\left(\check{\mathcal{T}}_m^{\rm emb}\setminus\left\{X_{\mathbf{i}},\mathbf{i}\in\mathbb{N}^{m+1}\right\}\right)=0,$$
where we note that $X_{\mathbf{i}}$ may still be a $\check{\mu}_{m+1}^{\rm emb}$-atom if $X_{\mathbf{i}j}=0$ for some $j\ge 1$. However, $\{X_{\mathbf{i}},\mathbf{i}\in\mathbb{N}^{m+1}\}$ is
countable, and all atom masses reduce to 0, so 
$$\check{\mu}_n^{\rm emb}\left(\check{\mathcal{T}}_m^{\rm emb}\right)=\sum_{\mathbf{i}\in\mathbb{N}^{m+1}}\check{\mu}_n^{\rm emb}(X_{\mathbf{i}})\longrightarrow 0\qquad\mbox{as $n\rightarrow\infty$,}$$
by dominated convergence. As countable unions of null sets are null, $\bigcup_{m\ge 0}\check{\mathcal T}_m^{\rm emb}$ is $\check{\mu}^{\rm emb}$-null a.s., and so $\check{\mu}^{\rm emb}$ must be supported by the limit points in the closure, which are only leaves. Properties (i) and (ii) of a continuum tree now follow. To see (iii), note that, since $L=\sup\{X_i\colon P_i>0,i\ge 1\}$ a.s., for all $x \in \check{\mathcal T}^{\rm emb} \setminus {\rm Lf}( \check{\mathcal T}^{\rm emb})$, we see from (\ref{expl}) that there is
$\mathbf{i}=i_1 \cdots i_n \in \mathbb U$ with 
$$x\in\check{X}_{\mathbf i}+\check{P}_{\mathbf{i}}^\beta[0,L_{\mathbf{i}}]e_{\mathbf{i}}.$$
But $x$ is not a leaf, so either $x=\check{X}_{\mathbf i}+\check{P}_{\mathbf{i}}^\beta L_{\mathbf{i}}e_{\mathbf{i}}$ is at the top of this interval, and there is $j\ge 1$ such 
that $X_{\mathbf{i}j}=L_{\mathbf{i}}$ and $P_{\mathbf{i}j}>0$, in which case the subtree above $x$ has positive $\check{\mu}^{\rm emb}_{n+1}$-mass, or $x$ is not at the
top, so there is $X_{\mathbf{i}j}$ beyond $x$ with positive $\check{\mu}_n^{\rm emb}$-mass. In either case, condition (iii) holds for $x$. 
\end{proof}

We now prove Corollary \ref{introcornew}, using an embedding of $(\mathcal T_k, \mu_k)$, $k \geq 0$, into $(\check{\mathcal T}^{\rm emb}, \check{\mu}^{\rm emb})$.

\begin{proof}[Proof of Corollary \ref{introcornew}] Consider $(\check{\mathcal T}_n^{\rm emb},\check{\mu}^{\rm emb}_n)$, $n\ge 0$, built from $\xi_{\mathbf{i}}$, $i\in\mathbb{U}$, as in Corollary \ref{recconcrt2}, with limit 
  $\check{\mathcal T}^{\rm emb}$, and mass measure $\check{\mu}^{\rm emb}$ as in Proposition \ref{mmreccon}. For simplicity, we use notation
  $$(\check{E}_{\mathbf{i}},\check{\mu}_{\mathbf{i}})=\left(\check{X}_{\mathbf i}+\check{P}_{\mathbf{i}}^\beta[0,L_{\mathbf{i}}]e_{\mathbf{i}}
															,\sum_{j\ge 1} \check{P}_{\mathbf{i}j}\delta_{\check{X}_{{\mathbf i}j}}\right),\qquad\mathbf{i}\in\mathbb{N}^n,n\ge 0,$$
  for the branch of $\check{\mathcal T}_n^{\rm emb}$ corresponding to $\xi_{\mathbf{i}}$. Now let 
  $(\overline{\mathcal T}_0,\overline{\mu}_0)=(\check{E}_{\varnothing},\check{\mu}_{\varnothing})$ and given $(\overline{\mathcal T}_j,\overline{\mu}_j)$, $0\le j\le k$,
  with $\overline{\mathcal T}_k\subseteq\check{\mathcal T}^{\rm emb}$ and $\overline{\mu}_k$ the push-forward of $\check{\mu}_k^{\rm emb}$ under the natural projection from 
  $\check{\mathcal T}$ to $\overline{\mathcal T}_k$, pick $\overline{J}_k$ from $\overline{\mu}_k$. If $\overline{J}_k=\check{X}_{\mathbf{i}j}$, remove
  $\overline{\mu}_k(\overline{J}_k)\delta_{\overline{J}_k}$ from $\overline{\mu}_k$ and add to $\overline{\mathcal T}_k$ the rescaled string of beads 
  $(\check{E}_{\mathbf{i}j},\check{\mu}_{\mathbf{i}j})$ to form $(\overline{\mathcal T}_{k+1},\overline{\mu}_{k+1})$. 

  Then $(\overline{\mathcal T}_k,\overline{\mu}_k,k\ge 0)$ has the same distribution as $({\mathcal T}_k,{\mu}_k,k\ge 0)$. Since 
  $\overline{\bigcup_{k\ge 0}\overline{\mathcal T}_k}\subseteq\check{\mathcal T}^{\rm emb}$ is compact, it remains to show that the inclusion is actually an equality. To see this,
  we employ a simpler variant of an argument of \cite[Proposition 22]{1}. Roughly, let $\epsilon>0$ and consider the connected components of
  $$\left\{x\in\check{\mathcal T}^{\rm emb}\colon {\rm ht}\left(\check{\mathcal T}^{\rm emb}_x\right)\le\epsilon\right\}$$  where $\check{\mathcal T}_x^{\rm emb}=\{\sigma \in \check{\mathcal T}^{\rm emb}: x \in [[\rho,\sigma]]\}$. Since $\check{\mathcal T}^{\rm emb}$ is compact, only finitely many attain height $\epsilon$. Each of these intersects $\check{E}_{\mathbf{i}}$ for some 
  $\mathbf{i}=i_1i_2\cdots i_n\in\mathbb{U}$, and each $\check{X}_{i_1\cdots i_j}$, $1\le j\le n$, is picked as some $\overline{J}_{k_j}$, $1\le j\le n$, after a geometric number of steps
  with parameter $\overline{\mu}_{k_{j-1}}(\check{X}_{i_1\cdots i_j})$, and the tree $\overline{\mathcal T}_{k_n}$ will intersect the component. When all components of height $\epsilon$ have
  been intersected, the Hausdorff-Prokhorov distance from $(\check{\mathcal T}^{\rm emb},\check{\mu}^{\rm emb})$ is below $\varepsilon$. This completes the proof.
\end{proof}

\section{Examples and applications} \label{appl}

\subsection{Self-similar CRTs and self-similar random weighted $\bR$-trees}

\noindent In \cite{33,34}, the construction of self-similar CRTs as the genealogies of fragmentation processes is carried out, as follows. 
\begin{itemize}
  \item Take Bertoin's \cite{51} self-similar exchangeable fragmentation process in the space of partitions of $\bN$. 
  \item Restrict the process to $[k]\subset\bN$ and construct $\bR$-trees with edge lengths, consistently for all $k\ge 1$. 
  \item Check Aldous's \cite{6} leaf-tightness criterion and estimate cover sizes to obtain a compact $\bR$-tree,
  \item and a mass measure as the weak limit of the uniform probability measure on the $k$ leaves, as $k$ tends to infinity, in a consistent embedding of the discrete $\mathbb{R}$-trees with edge lengths. 
\end{itemize}
Specialising our construction of Theorem \ref{constr2} (and Theorem \ref{uniquefix}) to the genealogies of fragmentation processes, including the limiting mass measure, amounts to the following.
\begin{itemize}
  \item Take Bertoin's \cite{51} self-similar exchangeable fragmentation process in the space of partitions of $\bN$. 
  \item Associate a generalised string with the blocks containing 1 and repeat recursively in all other blocks.
  \item Recursively build an $\bR$-tree which is the fixpoint of a recursive distribution equation,
  \item and construct the mass measure as the weak limit of the approximating mass measures.
\end{itemize} 
Let us make our construction more precise. Denote by $\cP_\bN$ the set of partitions of $\bN$. Exchangeable partitions $\pi=\{\pi_i,i\ge 1\}$ have \em asymptotic frequencies \em 
$|\pi_i|=\lim_{n\rightarrow\infty}n^{-1}\#\pi_i\cap[n]$, $i\ge 1$. Listing elements of any subset $C=\{c_j,j\ge 1\}\subseteq\bN$ in increasing order $c_1<c_2<\cdots$, we define $\pi\circ C=\{\{c_j,j\in\pi_i\},i\ge 1\}$. 

Bertoin \cite{51} defined a $\cP_\bN$-valued {\it self-similar exchangeable fragmentation process} $\Pi=(\Pi(t),t\ge 0)$ of index $-\beta$ to be a Markov 
process with $\Pi(0)=\{\bN\}$, which has a stochastically continuous {\em mass process} $|\Pi|$ of asymptotic frequencies, and satisfies the branching property that given 
$\Pi(t)=\{\pi_i,i\ge 1\}$, the process $\Pi(t+\cdot)$ has the same distribution as $(\bigcup_{i\ge 1}\Pi^{(i)}(r|\pi_i|^{-\beta})\circ\pi_i,r\ge 0)$, where $\Pi^{(i)}$, $i\ge 1$ is 
a family of independent copies of $\Pi$. In particular, blocks may fragment over time and if $\beta>0$, blocks with smaller mass do so more quickly.

Bertoin \cite{51} found that the masses $|\Pi_{(k)}(t)|$ of the block $\Pi_{(k)}(t)$ containing $k$, $t\ge 0$, form decreasing $\beta$-self-similar 
Markov processes that die in finite time. Indeed, $L=\inf\{t\ge 0\colon|\Pi_{(1)}(t)|=0\}$ has moments $\bE(L^p)<\infty$ for all $p\ge 0$. In \cite{Ber-book}, Bertoin 
showed an extended branching property at stopping lines such as $D_{1,i}=\inf\{t\ge 0\colon\Pi_{(1)}(t)\neq\Pi_{(i)}(t)\}$, $i\ge 2$, which was used in \cite{35} to 
study {\em spinal partitions} $\Gamma=\{\Pi_{(i)}(D_{1,i}),i\ge 2\}$, under some assumptions that are not actually needed for the present discussion. Let us denote by 
$(P_j,j\ge 1)$ the ranked masses of $\Gamma$, by $\Gamma_j$ the corresponding block of the spinal partition and by $X_j$ the
corresponding time $D_{1,i}$ so that $\Gamma_j=\Pi_{(i)}(D_{1,i})$, $j\ge 1$. We capture the singleton blocks lost when ranking in a measure 
$\Lambda([a,b])=\lim_{n\rightarrow\infty}n^{-1}\#\{i\in[n]\colon\{i\}\in\Gamma,D_{1,i}\in[a,b]\}$, $0\le a\le b\le L$. 

Then, $\xi_\varnothing^\bullet=([0,L],(X_j)_{j\ge 1},(P_j)_{j\ge 1},\Lambda)$ is a random generalised string. The extended branching property yields that conditionally given
$\xi_\varnothing^\bullet$, the blocks $\Gamma_j$, $j\ge 1$, evolve according to independent copies $\Pi^{(j)}$ of $\Pi$, with mass scaled by $|\Gamma_j|$ and time by $|\Gamma_j|^\beta$. We use 
the independent copies $\Pi^{(j)}$ of $\Pi$ to define independent copies $\xi_j^\bullet$, $j\ge 1$, of $\xi_\varnothing^\bullet$, and recursively an RTF $(\xi_\mathbf{i}^\bullet,\mathbf{i}\in\mathbb{U})$. 

Then the weighted compact $\bR$-tree $(\check{\mathcal T},\check{\mu})$ constructed in Corollary \ref{recconcrt2} and Proposition \ref{mmreccon} from the RTF
$((\xi_\mathbf{i}^\bullet,\mathbf{i}\in\mathbb{U});\phi_\beta)$ describes the genealogy of $\Pi$ since $\Pi$ uses the same scaling for time and block masses as the
construction of $(\check{\mathcal T},\check{\mu})$ for distances and atom masses/measures. Indeed, the fixpoint relation of $\check{\mathcal T}$ in terms of the 
generalised string $\xi=\xi^\bullet_\varnothing$ in Theorem \ref{uniquefix} is a version of the spinal decomposition theorem \cite[Proposition 4(ii)]{35} for self-similar
CRTs. By the spinal decomposition theorem and uniqueness in Theorem \ref{uniquefix}, the trees $(\check{\mathcal T},\check{\mu})$ are the corresponding self-similar CRTs 
of \cite{33}, and the other self-similar weighted compact $\bR$-trees of \cite{34}, by a straightforward generalisation of the spinal decomposition theorem. 

\begin{example}[Stable trees]\label{example}\rm The most prominent family of self-similar CRTs are stable trees \cite{17}. They arise as scaling limits of conditioned Galton-Watson 
trees \cite{24}, whose offspring distribution is in the domain of attraction of a $\theta$-stable law for $\theta\in(1,2]$, including the Brownian CRT for $\theta=2$. 

It was shown in \cite[Corollary 10]{35} that the (``fine'') spinal partition $\Gamma$ is sampled from a 
${\rm PD}(1/\theta,1-1/\theta)$ mass partition $(P_i,i\ge 1)$, and that the so-called ``coarse'' spinal partition $\Gamma^*=\{\bigcup_{j\ge 1\colon X_j=X_i}\Gamma_j,i\ge 1\}$ is a 
$(\theta-1,\theta-1)$-coagulation of $\Gamma$, in the sense of \cite{cmc,5}. Still by \cite[Corollary 10]{35}, $\Gamma^*$ is sampled from a ${\rm PD}(1-1/\theta,1-1/\theta)$ mass partition $(Q_m,m\ge 1)$, 
but also has a natural spinal order represented in an exchangeable $(1-1/\theta,1-1/\theta)$-interval partition in the sense of \cite{3}, translated into a $(1-1/\theta,1-1/\theta)$-string of
beads in \cite{1}. The spinal order is an independent uniform random order, and spinal distances are given by a sequence of independent 
uniform random variables $(U_m,m\ge 1)$, relative to the length $L$ of the spinal string. Fragmentation-coagulation duality \cite{5} applies independently of $(U_m,m\ge 1)$ and yields the representation of
$([0,L],(P_i)_{i\ge 1},(X_i)_{i\ge 1})$ given in Definition \ref{introdef}, for $\beta=1-1/\theta$.

The case $\theta=2$ of the Brownian CRT is binary and has a simpler $\Xi_{\rm s}$-valued $(1/2,1/2)$-string of beads as its spine, as identified explicitly in \cite[Proposition 14(b)]{1}, implicitly or expressed in different formalism in various previous papers. E.g., the associated spinal decomposition is closely related to decompositions of Brownian excursions in terms of Brownian bridges, see \cite[Section 3.3]{37} and \cite{BeP-94,AMP}.    
\end{example}

In particular, we can now apply Theorem \ref{uniquefix} in the case of $\beta$-generalised strings of beads to deduce Theorem \ref{fixp}.

One advantage of our construction is that the exchangeability of leaves is irrelevant. Lack of exchangeability was a hurdle in bead splitting processes based on general regenerative strings of beads. In 
\cite{37}, this problem was overcome by embedding into self-similar CRTs as constructed in \cite{33}. Here, we obtain compact CRT limits directly from Corollary \ref{introcornew}, indeed we obtain compact random 
weighted $\bR$-trees in much higher generality, as demonstrated in Corollary \ref{introcornew} and Theorems \ref{constr2} and \ref{constr3}.

\subsection{Genealogical trees of growth fragmentations}

Lack of exchangeability is also a feature in bead-splitting processes associated with Bertoin's genealogical construction of self-similar growth fragmentations \cite{53}. Specifically, Bertoin's
starting point is a positive $\beta$-self-similar Markov process $Z$ starting from 1 (with no positive jumps in \cite{53}, but the exclusion of positive jumps is not essential, see \cite{54}) with finite 
lifetime, which, via Lamperti's transform, can be constructed from a (spectrally negative) L\'evy process that drifts to $-\infty$. For simplicity, let us assume that
\begin{equation} \exists_{q>0}\quad	\kappa(q)=-k+\frac{1}{2}\sigma^2q^2+bq+\int_{\bR}\left((e^y)^q-1-q(1-e^y)+(1-e^y)^q\mathbbm{1}_{\{y<0\}}\right)\nu(dy)<0,
\label{locfincond}
\end{equation}
where $(k,b,\sigma^2,\nu)$ are the characteristics of the underlying L\'evy process $Y$. If $\kappa(q)>0$ for all $q>0$, the growth fragmentation is explosive \cite{57},
and the remaining case is more delicate, and our results do not apply, certainly not directly. Now recall notation $\bU=\bigcup_{n\ge 0}\bN^n$ for the infinite
Ulam-Harris tree. Let $(Z_u,u\in\bU)$ be an independent identically distributed family of copies of $Z$. We define the self-similar growth fragmentation, as follows.
\begin{itemize}\item The process $\cZ_\varnothing=Z_\varnothing$ is the evolution of the generation-0 fragment. 
  \item Recursively, for each $u\in\bU$, the ranked sequence $P_{uj}=|\Delta\cZ_u(X_{uj})|\mathbbm{1}_{\{\Delta\cZ_u(X_{uj})<0\}}$, $j\ge 1$, of negative jumps of $\cZ_u$,
    represent the fragments of the next generation, whose evolution is scaled in fragment size and time as $\cZ_{uj}=(P_{uj}Z_{uj}(P_{uj}^{-\beta}t),t\ge 0)$, $j\ge 1$.
  \item The process $\mathbf{Z}(t)=(\cZ_u(t-b_u),u\in\bU\colon b_u\le t<b_u+\zeta_u)^\downarrow$ of decreasing rearrangements of fragment sizes at time $t\ge 0$, is called a \em self-similar growth fragmentation of index
    $-\beta$\em, where $b_u=\sum_{i=1}^{|u|}X_{u_1\ldots u_i}$ and $\zeta_u=\inf\{t\ge 0\colon\cZ_u(t)=0\}$ denote the fragment birth time and fragment lifetime, 
    respectively, and where $(\cdot)^\downarrow$ denotes the decreasing rearrangement of the collection $(\cdot)$.
\end{itemize}  

This notion turns out to generalise the notion of a binary self-similar mass fragmentation process $|\Pi|^\downarrow$, which is obtained when the L\'evy process $Y$ is
the negative of a subordinator. By \cite[Lemma 3]{53}, the condition (\ref{locfincond}) is what is needed for Theorem \ref{uniquefix} to apply.

\begin{corollary}[Genealogy of growth fragmentations]\label{growthfrag} Let $(\mathbf{Z}(t),t\ge 0)$ be a self-similar growth fragmentation of index $-\beta$ that 
  satisfies (\ref{locfincond}) for some $q>0$ and $\beta\in(0,q]$. Then the genealogical trees $\check{\cT}_n$ up to generation $n$ based on (already rescaled) 
  $\Xi$-valued strings $(\zeta_u,(P_{uj},j\ge 1),(Z_{uj}(0),j\ge 1))$, $u\in\bU$, obtained from the construction of the growth fragmentation, converge in the 
  Gromov-Hausdorff sense to a compact limiting $\bR$-tree $\check{\mathcal T}$.
\end{corollary}
\begin{proof}
  We apply \cite[Lemma 3]{53}. In our notation, \cite[Lemma 3]{53} states that
  $$\varphi(q):=\bE\left[\sum_{j\ge 1}P_j^q\right]<1\qquad\mbox{if and only if}\qquad\kappa(q)<0,$$
  so the assumption of Theorem \ref{uniquefix} is satisfied for $p=q/\beta$. For the moment condition on $L=\zeta_\varnothing$ note that 
  $$\bE[L]=\bE\left[\int_0^\infty\exp(\beta Y_s)ds\right]=-\frac{1}{\psi(\beta)}<\infty\qquad\mbox{if }\psi(\beta):=\frac{1}{s}\log(\bE[\exp(qY_s)])<0.$$
  Since $\kappa-\psi\ge 0$, $\kappa(q)<0$ implies $\psi(q)<0$ and hence $\psi(\beta)<0$ by convexity. For higher moments of $L$ we refer to \cite[Proposition 3.1]{CPY97}
  to see that given $\bE[L]<\infty$, we have that $\psi(q)<0$ implies $\bE[L^{q/\beta}]<\infty$, as required.
  We conclude by applying Theorem \ref{constr3}. 
\end{proof}

%In our notation,  $(P_j,j\ge 1)$ for the ,
%{\tt Do we need $p\ge 1$? If so, maybe it's best to add the restriction as $\beta\le q$ to the statement of the corollary here.}

\begin{example}\label{exgftheta}\rm Corollary \ref{growthfrag} applies to the growth fragmentations of index $1-\theta$, associated with the Brownian map \cite{54,55}, in the case
$\theta=3/2$, in the family
$$\kappa_\theta(q)=\frac{\cos(\pi(q-\theta))}{\sin(\pi(q-2\theta))}\cdot\frac{\Gamma(q-\theta)}{\Gamma(q-2\theta)},\quad\theta<q<2\theta+1,\quad\mbox{with }\kappa_\theta(q)<0\iff\theta+1/2<q<\theta+3/2.$$
They were proposed by \cite{54} in connection with certain Boltzmann planar maps and the stable maps of \cite{LGM}, when $\theta\in(1,3/2]$. In the theory of random maps, the
present developments starting from positive initial mass $1$ or $x>0$ correspond to maps with a boundary (also called disks). In the case of the Brownian map without a
boundary, genealogical trees have been studied under the name of ``Brownian cactus'' \cite{CLM13}. When the index of self-similarity $1-\theta$ is changed to $-\theta$, 
the wider parameter range $\theta\in(1/2,3/2]$ is considered in \cite{54}. Again, this is within the framework of Corollary \ref{growthfrag}. 
\end{example}

Our methods establish moments for the height of the tree, which is the extinction time of the growth fragmentation, in fact we obtain $q/\beta$-moments for all $q>0$ with $\kappa(q)<0$.  

\begin{corollary}\label{corgfmom} In the setting of Corollary \ref{growthfrag}, the height of the tree $\check{\mathcal T}$ of index $-\beta$ has finite moment 
  $\mathbb E[{\rm ht}(\check{\mathcal T})^p]< \infty$ of order $p$ for all $p<\sup\{q>0\colon\kappa(q)<0\}/\beta$.
\end{corollary}
\begin{proof} This follows from Corollary \ref{growthfrag} (and its proof) and from Corollary \ref{heightmom}.  
\end{proof}

In \cite[Corollary 4.5]{54}, this was proved in the absence of positive jumps (i.e.\ when the L\'evy measure $\nu$ of $Y$ is concentrated on $(-\infty,0)$). In fact, they obtained a more precise tail behaviour of the extinction time of the
growth fragmentation, which is the height of $\check{\mathcal T}$. The growth fragmentation associated with the Brownian map has no positive jumps, so their result
already establishes finite moments of order $p<2$.

In the cases $\theta\in(1/2,3/2)$, the growth fragmentations of Example \ref{exgftheta} have positive jumps, since \cite{54} show that the L\'evy measure $\nu_\theta$
underlying $\kappa_\theta$ is the push-forward of the measure on $(1/2,\infty)$ with Lebesgue density
$$\frac{\Gamma(\theta+1)}{\pi}\left(x^{-\theta-1}(1-x)^{-\theta-1}1_{\{1/2<x<1\}}+\sin(\pi(\theta-1/2))x^{-\theta-1}(x-1)^{-\theta-1}1_{\{x>1\}}\right)$$
under the map $x\mapsto\log(x)$. Therefore, \cite[Corollary 4.5]{54} does not apply for $\theta\in(1/2,3/2)$, while Corollary \ref{corgfmom} yields finite moments of the 
height of $\check{\mathcal T}$, and hence for the extinction time of the growth fragmentation, as well as any associated random map, up to order 
$p<(\theta+3/2)/(\theta-1)=1+5/2(\theta-1)$ in the case $\beta=\theta-1$, $\theta\in(1,3/2]$, and up to order $p<(\theta+3/2)/\theta=1+3/2\theta$ in the case $\beta=\theta$, $\theta\in(1/2,3/2]$.

%Associate the string of beads of length $L=\inf\{t\ge 0\colon\lambda(t)=0\}$ with mass measure $\mu_0((t,L])=\lambda(t)$, $t\ge 0$.

%Recall from the Introduction the definition of Bertoin's \cite{53} self-similar (Markovian) growth fragmentation process $\mathbf{Z}=(\mathbf{Z}(t), t \geq 0)$ derived from a positive self-similar Markov process $Z$. We then obtain i.i.d. generalised strings $\xi_{u}=(\zeta_u, (x_{uj})_{j \geq 1}, (P_{uj})_{j \geq 1}, 0)$ to which we can apply our recursive construction from Theorem \ref{constr3}.

\subsection{Hausdorff dimension of fixpoint trees $\check{\mathcal{T}}$}

We will show that the Hausdorff dimension of the set of leaves of $\check{\mathcal{T}}$ is generally $q^*/\beta$, where we define $q^*=\inf\{q>0\colon\varphi(q)<1\}$ and 
$\varphi(q)=\mathbb{E}[\sum_{i\ge 1}P_i^q]$. We will need a technical assumption of $\mathbb{E}[L^{-q/\beta}]<\infty$ to avoid very short strings that may be able to pack 
leaves very close together and potentially cause a drop in Hausdorff dimension. We do not know if this drop in dimension actually ever happens. We also need to address two trivialities. The first is that in the case
$\mathbb{P}(X_i=0\mbox{ for all }i\ge 1\mbox{ with }P_i>0)=1$, the tree $\check{\mathcal{T}}$ will be a star tree with countably many leaves.  The second is that in the 
case where $\mathbb{P}(P_1=0)>0$, the ``extinction event'' $\mathcal{E}=\{\inf\{n\ge 1\colon\check{P}_{\mathbf{i}}=0\mbox{ for all }\mathbf{i}\in\mathbb{N}^n\}<\infty\}$ 
has positive probability, and on $\mathcal{E}$, the tree $\check{\mathcal{T}}$ will have only finitely many leaves. We refer to $\mathcal{E}^c$ as the event of 
non-extinction.

\begin{theorem}\label{dim} Consider $(\check{\mathcal{T}},\check{\mu})$ in the setting of Theorem \ref{constr3}. Assume further that $\mathbb{E}[L^{-q/\beta}]<\infty$ and
  that $\mathbb{P}(X_i>0\mbox{ for all }i\ge 1\mbox{ with }P_i>0)=1$.
  Let $\varphi(q)=\mathbb{E}[\sum_{i\ge 1}P_i^q]$ and $q^*=\inf\{q>0\colon\varphi(q)<1\}$. Then the set ${\rm Lf}(\check{\mathcal{T}})$ of leaves of $\check{\mathcal{T}}$ has 
  Hausdorff dimension ${\rm dim}_{\rm H}({\rm Lf}(\check{\mathcal{T}}))=q^*/\beta$, and $\check{\mathcal{T}}$ itself has Hausdorff dimension 
  ${\rm dim}_{\rm H}(\check{\mathcal{T}})=\max\{q^*/\beta,1\}$ a.s., on the event $\mathcal{E}^c$ of non-extinction.
\end{theorem}

The case $\mathbb{P}(X_i=0\mbox{ for some }i\ge 1\mbox{ with }P_i>0)>0$, while not included in the theorem nor in the trivialities discussed above, can be reduced to the case of the theorem. Specifically, 
consider the RTF $(\xi_{\mathbf{i}},\mathbf{i}\in\mathbb{U})$ in such a case. For all $\mathbf{i},\mathbf{j}=j_1\ldots j_m\in\mathbb{U}$ with $X_{\mathbf{i}j_1}>0$,  
$X_{\mathbf{i}j_1j_2}=\cdots=X_{\mathbf{i}j_1\ldots j_m}=0$, the string $\xi_{\mathbf{i}\mathbf{j}}$ is effectively attached to position $\check{X}_{\mathbf{i}j_1}$ of
an earlier generation, with scaling factor $\check{P}_{\mathbf{i}\mathbf{j}}/\check{P}_{\mathbf{i}}$. By re-assigning all such atoms to $\xi_{\mathbf{i}}$, we can 
construct the same $\check{\mathcal{T}}$ from RTFs to which the theorem applies. We leave the details to the reader.

We also stress that while the assumption that there is $q$ with $\varphi(q)<1$ may be regarded as a Malthusian assumption in the language of branching process theory, 
we do not actually assume $\varphi(q^*)=1$. It is possible that $\varphi(q^*)=\infty$ or that $\varphi(q^*)<1$. Indeed, under our assumptions, $\varphi$ is always finite 
on a sub-interval of $(0,\infty)$. This interval may be bounded above if $P_1$ lacks higher moments and bounded away from 0 if $P_i\rightarrow 0$ too slowly as 
$i\rightarrow\infty$. As an extreme example, consider $P_1$ with probability density function proportional to $x^{-2}(\log x)^{-2}1_{\{x\ge\epsilon\}}$ and 
$P_i=P_1/2i\log(i)^2$, $i\ge 2$. Then $\varphi(q)<\infty$ if and only if $q=1$, and for $\epsilon$ small enough, $\varphi(1)<1$.

The big steps of the proof are the same as in the self-similar case in the sense of Haas and Miermont \cite{33} and Stephenson \cite{34}. In particular, we will adapt 
Stephenson's method of biasing the tree for the sampling of two leaves from suitable finite random measures on the trees. While \cite{33}, and \cite{34} to some  
extent, worked with infinite dislocation measures of continuous-time fragmentation processes, we work exclusively with the discrete branching structure along the 
generations in $\mathbb{U}$. Such a classical branching processes approach via Malthusian martingales was used very recently in the special case of growth fragmentations \cite{54} to construct an \em intrinsic area measure \em on the boundary $\partial\mathbb{U}$ of $\mathbb{U}$.

Let us sketch our argument to point out some of the simplifications, as well as the approximation method for the lower bound, which is 
new as unlike \cite{34}, we have no automatic control of negative moments of the heights of random leaves. Indeed, the relevant assumption of Theorem \ref{dim} is
made directly in terms of the most basic data, the length $L$ of the string, while locations $X_i$, $i\ge 1$, of small atoms may accumulate at the left end-point. 

\begin{lemma}[Upper bound]\label{upper} In the setting of Theorem \ref{constr3}, with $q^*$ as in Theorem \ref{dim}, we have 
  ${\rm dim}_{\rm H}({\rm Lf}(\check{\mathcal{T}}))\le q^*/\beta$.
\end{lemma}
\begin{proof} Let $\check{P}_{\mathbf{i}}$, $\mathbf{i}\in\mathbb{U}$, be as in (\ref{pcheck}), and let $\epsilon>0$. Consider the branching process 
  $(\check{P}_{\mathbf{i}},\mathbf{i}\in\mathbb{N}^n)$, $n\ge 0$, and the stopping line  
  $$\mathcal{L}_\epsilon=\{i_1\ldots i_n\in\mathbb{U}\colon\check{P}_{i_1}>\epsilon,\ldots,\check{P}_{i_1\ldots i_{n-1}}>\epsilon,\check{P}_{i_1\ldots i_n}\le\epsilon\},$$
  which is a ``finite stopping line'' in the sense that it a.s. stops after a finite number of generations (for instance as $\check{\mathcal{T}}$ is compact). This 
  stopping line generates the system of frozen cells considered in \cite[Proposition 2.5]{54} in connection with the Malthusian martingale, when $\varphi(q^*)=1$.  
  In our generality, we can see for all $q\ge q^*$ with $\varphi(q)<1$, that $\mathbb{E}[\sum_{\mathbf{i}\in\mathcal{L}_\epsilon}\check{P}_{\mathbf{i}}^{q}]<1$, and by
  the (elementary) extended branching property at $\mathcal{L}_\epsilon$ for the discrete-time branching process, the heights $H_{\mathbf{i}}$ of the subtrees are i.i.d.
  heights scaled by $\check{P}_{\mathbf{i}}^\beta$ such that 
  $$\mathbb{E}\left[\sum_{\mathbf{i}\in\mathcal{L}_\epsilon}H_{\mathbf{i}}^{q/\beta}\right]
    =\sum_{\mathbf{i}\in\mathbb{U}}\mathbb{E}\left[1_{\{\mathbf{i}\in\mathcal{L}_\epsilon\}}\check{P}_{\mathbf{i}}^{q}\right]
                                   \mathbb{E}\left[H_\varnothing^{q/\beta}\right]
    <\mathbb{E}\left[H_\varnothing^{q/\beta}\right]<\infty$$
  by Corollary \ref{heightmom}, as $\mathbb{E}[L^{q/\beta}]<\infty$ for $q>q^*$ sufficiently small. Since the balls $\overline{B}(\check{X}_{\mathbf{i}},H_{\mathbf{i}})$ 
  of radius $H_{\mathbf{i}}$ around the roots $\check{X}_{\mathbf{i}}\in\check{\mathcal{T}}$ of the subtrees, $\mathbf{i}\in\mathcal{L}_\epsilon$, together with the 
  countable number of singleton leaves at the ends of strings, form refining covers of ${\rm Lf}(\check{\mathcal{T}})$ as $\varepsilon\downarrow 0$, the Hausdorff 
  dimension is at most $q/\beta$, for some $q\downarrow q^*$, as required.   
\end{proof}

This establishes the upper bound claimed in Theorem \ref{dim}.

\begin{lemma}[Lower bound]\label{lower} In the setting of Theorem \ref{dim}, suppose further that $X_i>\epsilon L$ for $1\le i\le N$ and $P_i=0$ for
  $i\ge N+1$, for some $N<\infty$. Then either $q^*=0$ or $\varphi(q^*)=1$, and ${\rm dim}_{\rm H}({\rm Lf}(\check{\mathcal{T}}))\ge q^*/\beta$, on the event
  $\mathcal{E}^c$ of non-extinction. 
\end{lemma} 
\begin{proof} First note that $\varphi(q)=\mathbb{E}[\sum_{1\le i\le N}P_i^q]$ is finite, continuous, convex, on an interval including $0$ with $\varphi(0)\le N$. Hence, either
  $\varphi(0)<1$ and $q^*=0$, or $\varphi(0)\ge 1$ and $q^*\ge 0$ with $\varphi(q^*)=1$ exists. If $q^*=0$, the claim is trivial. So assume $q^*>0$ with $\varphi(q^*)=1$. Note that
  for $q>q^*$ with $\varphi(q)<1$, we have $r:=q/q^*>1$, and so
  $$\mathbb{E}\left[\left(\sum_{1\le i\le N}P_i^{q^*}\right)^r\;\right]\le N^r\mathbb{E}[(P_1^{q^*})^r]<N^r<\infty.$$
  By a slight generalisation of \cite[Theorem 1.1 and (1.20)]{Ber-book} to include the possibility $\bP(P_i>1)>0$, the process 
  $M_n=\sum_{\mathbf{i}\in\mathbb{N}^n}\check{P}_{\mathbf{i}}^{q^*}$ is a uniformly integrable martingale, whose terminal value $M_\infty$ is a.s. strictly positive on the
  event $\mathcal{E}^c$ of non-extinction. In the same way, we find a.s. martingale limits  
  $M_\infty^{(\mathbf{j})}:=\lim_{n\rightarrow\infty}\sum_{\mathbf{i}\in\mathbb{N}^n}\check{P}_{\mathbf{j}\mathbf{i}}^{q^*}/\check{P}_{\mathbf{j}}^{q^*}$ and note that,
  by construction, 
  \begin{equation}\label{inde}M_\infty^{(\mathbf{j})}=\sum_{i\ge 1}P_{\mathbf{j}i}^{q^*}M_\infty^{(\mathbf{j}i)},\qquad\mbox{a.s., with }(P_{\mathbf{j}i})_{i\ge 1},M_\infty^{(\mathbf{j}i)},i\ge 1,\mbox{ independent.}
  \end{equation}
  By specifying $\mu^*(\mathcal{S}_{\mathbf{j}})=\check{P}_{\mathbf{j}}^{q^*}M_\infty^{(\mathbf{j})}$ for the subtree $\mathcal{S}_{\mathbf{j}}\subset\check{\mathcal{T}}$ 
  rooted at $\check{X}_{\mathbf{j}}$ and corresponding to $\check{P}_{\mathbf{j}}$, we define a measure $\mu^*$ on $\check{\mathcal{T}}$. More precisely, we specify (using
  the embedding (\ref{expl}) into $l^1(\mathbb{U})$ to be definite)
  $$m(\check{X}_{\mathbf{i}}+xe_\mathbf{i})=\sum_{j=1}^N\check{P}_{\mathbf{i}j}^{q^*}M_\infty^{(\mathbf{i}j)}1_{\{\check{P}_{\mathbf{i}}^\beta X_{\mathbf{i}j}\ge x\}}\quad\mbox{for each }0<x\le \check{P}_{\mathbf{i}}^\beta L_{\mathbf{i}},$$
  and apply \cite[Proposition 2.7]{34} to obtain the existence and uniqueness of the random measure $\mu^*$. Clearly, the countably many rescaled strings of 
  $\check{\mathcal{T}}$ carry zero $\mu^*$-mass a.s., so $\mu^*$ is supported by ${\rm Lf}(\check{\mathcal{T}})$ a.s.. We now check that for all $0<\gamma<q^*/\beta$
  \begin{align*}&\mathbb{E}\left[\int_{\check{\mathcal{T}}}\int_{\check{\mathcal{T}}}\frac{1}{d(v,v^\prime)^\gamma}\mu^*(dv)\mu^*(dv^\prime)\right]\\
    &\le\sum_{n\ge 0}\sum_{\mathbf{i}\in\mathbb{N}^n}2\sum_{1\le j<j^\prime\le N}\mathbb{E}\left[\int_{\check{\mathcal{T}}}\int_{\check{\mathcal{T}}}1_{\{v\in\mathcal{T}_{\mathbf{i}j},v^\prime\in\mathcal{T}_{\mathbf{i}j^\prime}\}}\frac{1}{d(v,\check{X}_{\mathbf{i}j})^\gamma}\mu^*(dv)\mu^*(dv^\prime)\right]\\
    &=2\sum_{n\ge 0}\sum_{\mathbf{i}\in\mathbb{N}^n}\bE\left[\check{P}_{\mathbf{i}}^{2q^*-\beta\gamma}\right]\sum_{1\le j<j^\prime\le N}\mathbb{E}\left[P_{\mathbf{i}j}^{q^*-\beta\gamma}P_{\mathbf{i}j^\prime}^{q^*}\right]\mathbb{E}\left[M_\infty^{(\mathbf{i}j^\prime)}\right]\mathbb{E}\left[\int_{\check{\mathcal{T}}_{\mathbf{i}j}}\frac{1}{d_{\mathbf{i}j}(v,\rho_{\mathbf{i}j})^\gamma}\mu_{\mathbf{i}j}^*(dv)\right],
  \end{align*}    
  by the independence in (\ref{inde}), where $(\check{\mathcal{T}}_{\mathbf{i}j},d_{\mathbf{i}j},\rho_{\mathbf{i}j})$ is the tree 
  constructed from the RTF $(\xi_{\mathbf{i}j\mathbf{j}},\mathbf{j}\in\mathbb{U})$ as in Corollary \ref{recconcrt2}, which has the same distribution as 
  $(\check{\mathcal{T}},d,\rho)$, and whose scaling by $\check{P}_{\mathbf{i}}^\beta P_{\mathbf{i}j}^\beta$ (and shifting in $l^1(\mathbb{U})$, as appropriate) is the 
  subtree $(\mathcal{S}_{\mathbf{i}j},d,\check{X}_{\mathbf{i}j})$ of $\check{\mathcal{T}}$. The analogue $\mu^*_{\mathbf{i}j}$ of $\mu^*$ is the push-forward of the 
  scaled restriction $\check{P}_{\mathbf{i}j}^{-q^*}\mu^*|_{\mathcal{S}_{\mathbf{i}j}}$ from $\mathcal{S}_{\mathbf{i}j}\subset\check{\mathcal{T}}$ to 
  $\check{\mathcal{T}}_{\mathbf{i}j}$ under the natural map. Since $\mathbb{E}[M_\infty^{(\mathbf{i}j^\prime)}]=1$, we obtain
  $$\mathbb{E}\left[\int_{\check{\mathcal{T}}}\int_{\check{\mathcal{T}}}\frac{1}{d(v,v^\prime)^\gamma}\mu^*(dv)\mu^*(dv^\prime)\right]
    \le\frac{2}{1-\varphi(2q^*-\beta\gamma)}(N-1)\varphi(2q^*-\beta\gamma)\frac{1}{\epsilon^\gamma}\mathbb{E}[L^{-\gamma}]<\infty,$$
  also using that $d(v,\rho_{\mathbf{i}j})\ge\epsilon L_{\mathbf{i}j}$ for $\mu^*_{\mathbf{i}j}$-a.e. $v\in\check{\mathcal{T}}_{\mathbf{i}j}$, by assumption. 
  The proof is complete, by an application of Frostman's lemma, as in \cite{33,34}, noting that $\gamma<q^*/\beta$ was arbitrary. 
\end{proof}

\begin{proof}[Proof of Theorem \ref{dim}] Since the upper bound was obtained in Lemma \ref{upper}, let us turn to the lower bound, which Lemma \ref{lower} only supplies
  under additional assumptions. Consider now the general setting of Theorem \ref{dim}. Also assume $q^*>0$, as the claimed lower bound
  is trivial otherwise. Let $\xi=([0,L],(X_i)_{i\ge 1},(P_i)_{i\ge 1})$ be a random generalised string from the RTF underlying $\check{\mathcal{T}}$. 
  
  Let $\epsilon\in[0,1)$ and $N\in\mathbb{N}$ and denote by $\xi^{(N,\epsilon)}$ a string with atom masses changed to $P_i1_{\{1\le i\le N,X_i>L\epsilon\}}$. Modifying 
  some offending locations of annihilated atoms, if necessary, and assuming $\epsilon\in(0,1)$, Lemma \ref{lower} applies to obtain subtrees 
  $\check{\mathcal{T}}^{(N,\epsilon)}\subset\check{\mathcal{T}}$, with ${\rm dim}_{\rm H}({\rm Lf}(\check{\mathcal{T}}^{(N,\epsilon)})\ge q^*_{N,\epsilon}/\beta$, where 
  $$q^*_{N,\epsilon}=\inf\{q>0\colon\varphi_{N,\epsilon}(q)<1\}\le q^*\qquad\mbox{and}\qquad
    \varphi_{N,\epsilon}(q)=\mathbb{E}\left[\sum_{i\ge 1}P_i^q1_{\{1\le i\le N,X_i>L\epsilon\}}\right]\le\varphi(q).$$ 
  By construction, ${\rm Lf}(\check{\mathcal{T}}^{(N,\epsilon)})\setminus{\rm Lf}(\check{\mathcal{T}})$ is at most countable (top ends of strings, if with positive
  probability, $\xi$ but not $\xi^{(N,\epsilon)}$ has an atom at $L$). Hence, ${\rm dim}_{\rm H}({\rm Lf}(\check{\mathcal{T}}))\ge q^*_{N,\epsilon}/\beta$ for all 
  $\epsilon\in(0,1)$, $N\in\mathbb{N}$, on the event $\mathcal{E}^c_{N,\epsilon}$ of non-extinction. 

  As $q^*>0$, and as $\varphi_{N,\epsilon}$ is continuous starting from the expected number $\varphi_{N,\epsilon}(0)$ of non-zero atoms, there will be $\epsilon_N>0$
  for all $N\ge 2$ such that $q_{N,\epsilon}^*>0$ for all $\epsilon<\epsilon_N$. For such $\epsilon<\epsilon_N$, $\varphi_{N,0}$ is convex, so that $q^*_{N,0}$ is an   
  isolated root of $\varphi_{N,0}-1$, and as $\epsilon\downarrow 0$, we find $q_{N,\epsilon}^*\uparrow q_{N,0}^*$. Similarly, $q_{N,0}^*\uparrow q^*$, as
  $N\rightarrow\infty$. Since also $\mathcal{E}^c=\bigcup_{N\ge 2,\epsilon>0}\mathcal{E}_{N,\epsilon}$, this completes the proof of the claim that 
  ${\rm dim}_{\rm H}({\rm Lf}(\check{\mathcal{T}}))=q^*/\beta$ on $\mathcal{E}^c$.   
  Since ${\rm dim}_{\rm H}(\check{\mathcal{T}}\setminus{\rm Lf}(\check{\mathcal{T}}))=1$, we then deduce that ${\rm dim}_{\rm H}(\check{\mathcal{T}})=\max\{q^*/\beta,1\}$,
  on $\mathcal{E}$.
\end{proof}

\subsection{Line-breaking constructions and binary embedding of the stable trees}

In Corollary \ref{introcornew} we obtained CRTs as limits of binary bead-splitting processes by embedding into the CRTs of Theorem \ref{introthm}. This includes the 
bead splitting process $(\cT_k,\mu_k)$, $k\ge 0$ based on $(1/2,1/2)$-strings of beads as discussed as the $\theta=2$ case of Example \ref{example}. If we drop the mass 
measures, the increments $\cT_{k+1}\setminus\cT_k$ of $(\cT_k,k\ge 0)$, are just isometric to intervals that we can all take successively from the half-line $[0,\infty)$.
In the case of the Brownian CRT, the sequence has a well-known autonomous description, which we can formulate as follows. 
\begin{example}[Aldous's line-breaking construction]\rm Consider the points $0<C_0<C_1<\cdots$ of an inhomogeneous Poisson process of intensity $tdt$ on the line
    $[0,\infty)$. Let $\cT_0$ be a one-branch
    tree of length $C_0$. For $k\ge 0$, to obtain $\cT_{k+1}$ conditionally given $\cT_k$, pick a point $J_k\in\cT_k$ from the normalised length measure on the branches
    and attach at $J_k$ a branch of length $C_{k+1}-C_k$. The trees converge, as $k\rightarrow\infty$ to the Brownian CRT $\cT$, when suitably represented in $\bT$ or 
    $\bT^{\rm emb}$. Equip $\cT$ in $\bT^{\rm emb}$ with the almost sure weak 
    limit of the normalised length measure on the branches of $\cT_k$, as $k\rightarrow\infty$. 
\end{example} 
We can similarly study processes as in Corollary \ref{introcornew} in the setting of the more general Theorem \ref{constr2}. We will here be particularly interested in the case of multifurcating branch points. 
\begin{corollary}[Multifurcating bead-splitting processes] Let $p\ge 1$, let $\xi=(\mathcal T_0, (X_i^{(0)})_{i \geq 1},  (P_i^{(0)})_{i \geq 1},\Lambda_0)$ be a random 
  generalised string of length $L$ with $\mathbb{E}[L^p]<\infty$. Let $\beta\in(0,\infty)$ such that $\mathbb{E}[\sum_{j\ge 1}P_j^{p\beta}]<1$. For $k \geq 0$, to obtain $(\mathcal T_{k+1}, (X_i^{(k+1)})_{i \geq 1},  (P_i^{(k+1)})_{i \geq 1},\Lambda_{k+1})$ conditionally given $(\mathcal T_{k}, (X_i^{(k)})_{i \geq 1},  (P_i^{(k)})_{i \geq 1},\Lambda_{k})$, pick an atom $(X_i^{(k)},P_i^{(k)})$ with probability proportional to $P_i^{(k)}$, $i\ge 1$, attach at $X_i^{(k)}\in\mathcal T_k$ an independent isometric copy of $\xi$ with metric rescaled by $(P^{(k)}_i)^\beta$ and measure/atom masses rescaled by $P^{(k)}_i$. Let $\mu_k=\Lambda_k + \sum_{i \geq 1}P_i^{(k)} \delta_{X_i^{(k)}}$. Then there exists a random weighted $\mathbb{R}$-tree $({\mathcal T},\mu)$ such that
$$ \lim \limits_{k \rightarrow \infty} \left({\mathcal T}_k,\mu_k\right)= \left({\mathcal T},\mu\right) \quad \text{a.s. in the Gromov-Hausdorff-Prokhorov topology on $\mathbb T_{\rm w}$}.$$
\end{corollary}
\begin{proof} The proof of Corollary \ref{introcornew} at the end of Section \ref{resultsstrings} is easily adapted.
\end{proof}
Goldschmidt and Haas \cite{12} studied line-breaking constructions of stable trees (without atoms on the branches). They are based on what they call the Mittag-Leffler Markov chain (MLMC) of parameter $\beta$, see \cite{12}, starting from the length of a $(\beta,\beta)$-string of beads.  
\begin{example}[Stable line-breaking construction]\rm Consider the MLMC $0<C_0<C_1<\cdots$. Let $\cT_0$ be a one-branch tree of length $C_0$, which
    has no branch points. For $k\ge 0$, to obtain $\cT_{k+1}$ conditionally given $\cT_k$ with branch points $v_i$ and weights $W_k^{(i)}$ so that total length plus sum
    of weights add up to $C_k$, select a branch point $v_i$ with probability proportional to $W_k^{(i)}$ or a branch between two branch points with probability
    proportional to its length. If a branch is selected, create a new branch point $v_i$ sampled from the length measure on the branch and select it. Attach at the 
    selected $v_i$ a branch of length $B_k(C_{k+1}-C_k)$, for an independent $B_k\sim{\rm Beta}(1,1/\beta-2)$, and increase the weight $W_k^{(i)}$ by 
    $(1-B_k)(C_{k+1}-C_k)$. The trees converge, as $k\rightarrow\infty$, to the stable tree $\cT$ of index $1/(1-\beta)$. 
\end{example}
Goldschmidt and Haas \cite{12} ask if there is a sensible way to associate a notion of ``length'' $W_k^{(i)}$ with the vertex $v_i$. A natural possibility is to make the branches longer by attaching a
branch of length $C_{k+1}-C_k$ instead, but this poses some questions. First, does this construction have a compact limit? Second, how do we distinguish the extra lengths from the lengths present in the
stable tree? Third, is this an interesting structure with further properties that makes this ``sensible''? In the context of the present paper, the fundamental question is how to turn branches into strings
of beads. While we fully address these questions in forthcoming work \cite{forth}, let us here construct the binary compact limiting CRT.

Specifically, we construct a binary tree $\mathcal T^\circ$ using Theorem \ref{introthm} based on i.i.d. isometric copies of a $\beta$-mixed string of beads, which we 
define and discuss in our final example:

\begin{example}\label{betamixed}\rm Let $\beta \in (0,1/2]$, and consider the function $\Psi_\beta: \widetilde{\Xi}_s \times \widetilde{\Xi}_s \times [0,1] \rightarrow  \widetilde{\Xi}_s$,
where, for $([0,\ell_1],\lambda_1),  ([0,\ell_2], \lambda_2) \in \widetilde{\Xi}_{\rm s}$ and $b \in [0,1]$, we define
\begin{equation} \left([0,\ell], \lambda \right) :=\Psi_\beta \left(\left([0,\ell_1], \lambda_1 \right), \left([0,\ell_2], \lambda_2\right), b \right) \label{stringmerged}
\end{equation}
via $\ell:=b^\beta \ell_1+(1-b)^\beta \ell_2$, and with $\ell':=b^\beta \ell_1$ the mass measure $\lambda$ on $[0,\ell]$ given by
\begin{equation} 
\lambda\left(\left[0,x\right]\right)=\begin{cases} b \cdot \lambda_1\left(\left[0,b^{-\beta}x\right]\right), &\text{ if } x \in \left[0,\ell'\right], \\
b+ \left(1-b\right) \cdot \lambda_2\left(\left[0,\left(1-b\right)^{-\beta}\left(x-\ell'\right)\right]\right), &\text{ if } x \in \left[\ell', \ell\right]. \end{cases}
\end{equation}
The string of beads $\Psi(\widetilde{\xi}_1,\widetilde{\xi}_2,B)$, where $\widetilde{\xi}_1$ and $\widetilde{\xi}_2$ are independent $(\beta, 1-2\beta)$- and 
$(\beta, \beta)$-strings of beads, respectively, and where $B \sim \rm{Beta}(1-2\beta, \beta)$ is independent, is called a \textit{$\beta$-mixed} string of beads. Since 
the lengths of $(\alpha,\theta)$-strings of beads as defined in \cite[Definition 4]{1} generalising Definition \ref{introdef} have moments of all orders, Theorem \ref{introthm} applies
to give a limiting CRT, which we denote by $\cT^\circ$, and whose height has moments of all orders, by Corollary \ref{heightmom}. 
\end{example}

While this example fits perfectly into the theory developed in this paper, $\mathcal{T}^\circ$ is not the binary compact limiting CRT we require. The modification is 
simple and points to a range of possible generalisations of the constructions presented in this paper, away from identical distribution of the random strings of beads. 
For ease of reference, we only present the result relevant for us in \cite{forth}, leaving any generalisations to the reader.  

\begin{prop} Consider $(\xi_{\mathbf{i}},\mathbf{i}\in\mathbb{U})$ such that $\xi_\varnothing$ is a $(\beta,\beta)$-string of beads independent of independent 
  $\beta$-mixed strings of beads $\xi_{\mathbf{i}}$, $\mathbf{i}\in\mathbb{U}\setminus\{\varnothing\}$. For $\mathbf{i}\in\mathbb{U}\setminus\{\varnothing\}$, let 
  $\mathcal{T}_{\mathbf{i}}^*$ be the tree constructed as in Proposition \ref{recconcrt}, but from the RTF $\{\xi_{\mathbf{i}\mathbf{j}},\mathbf{j}\in\mathbb{U}\}$, and 
  let $\mathcal{T}^*:=\mathcal{T}_\varnothing^*=\phi_\beta(\xi_{\varnothing},\mathcal{T}_j^*,j\ge 1)$. Then
  \begin{equation}\mathcal{T}_{\mathbf i}^*=\phi_\beta(\xi_{\mathbf i},\mathcal{T}_{{\mathbf i}j}^*,j\ge 1)\qquad\mbox{for all }\mathbf{i}\in\mathbb{U}.
  \end{equation}    
  Furthermore, we can equip $\mathcal{T}^*$ with a mass measure $\mu^*$ as in Proposition \ref{mmreccon} and Corollary \ref{eqmeas}.   
\end{prop} 
\begin{proof} Most of this follows straight from the previous constructions. Note that $\mathcal{T}_{\mathbf{i}}$ has the same distribution as $\mathcal{T}^\circ$ in 
  Example \ref{betamixed} for all $\mathbf{i}\in\mathbb{U}\setminus\{\varnothing\}$. Let us check that $\mathcal{T}_\varnothing$ does not collapse to a point tree 
  (recall that the point tree is what $\phi_\beta$ assigns if compactness fails when grafting rescaled $\mathcal{T}_i^*$ onto $\xi_\varnothing$, $i\ge 1$). The trees attached to the atoms of $\xi_\varnothing$ on $[0,L_\varnothing]$ have 
  heights $P_i^\beta{\rm ht}(\mathcal{T}_i^*)$, $i\ge 1$. Hence for all $\epsilon>0$,
  $$\sum_{i\ge 1}\mathbb{P}\left(P_i^\beta{\rm ht}(\mathcal{T}_i^*)>\epsilon\right)\le\epsilon^{-p}\sum_{i\ge 1}\mathbb{E}\left[P_i^{p\beta}\right]\mathbb{E}\left[{\rm ht}(\mathcal{T}_i^*)^p\right]\le\epsilon^{-p}\mathbb{E}\left[{\rm ht}(\mathcal{T}^\circ)^p\right]<\infty,$$
  where we choose $p$ such that $p\beta\ge 1$ and recall that $\mathcal{T}^\circ$ has all moments finite, as noted in Example \ref{betamixed}. By the first Borel-Cantelli
  lemma and the compactness of $\mathcal{T}_i^*$, $i\ge 1$, and of $\xi_\varnothing$, we conclude that the tree after grafting is compact, as required.
\end{proof}

%\begin{theorem}[Binary embedding of the stable tree, \cite{forth}] Let $(\mathcal T_i^*, i \geq 1)$ be i.i.d. copies of the tree $\mathcal T$ constructed via Theorem \ref{introthm} applied to i.i.d. isometric copies of a $\beta$-mixed strings of beads. 
%Consider an independent $(\beta, \beta)$-string of beads $([0,K], \sum_{i \geq 1} P_i \delta_{x_i})$, and attach to $x_i$ the tree $\mathcal T_i^*$ with distances rescaled by $P_i^\beta$. Then the tree obtained from $\mathcal T^*$ by contracting all marked components to a single branch point is a stable tree of index $1/(1-\beta) \in (1,2]$.
%\end{theorem}

\section*{Acknowledgements} 

We thank Alex Watson, Igor Kortchemski and Jean Bertoin for discussions about growth fragmentations, and Jean Bertoin and B\'en\'edicte Haas for asking us about the  
Hausdorff dimensions of the new trees.

\bibliographystyle{acm}
\bibliography{BinEmb}	 

\nocite{*}

\end{document}